\documentclass[a4paper,11 pt]{article}

\pdfoutput=1
\usepackage{color}
\usepackage{amsthm}
\usepackage{amsfonts}
\usepackage{amssymb}
\usepackage{amsmath}
\usepackage{graphicx}
\usepackage{algorithm}
\usepackage{algpseudocode}
\usepackage{caption}
\usepackage{enumerate}

\usepackage{mathptmx}
\usepackage[mathptmx]{flexisym}
\usepackage{breqn}
\usepackage{hyperref}

\usepackage{url}
\usepackage[affil-sl]{authblk}

\usepackage[symbol*]{footmisc}

\newcommand{\eps}{\varepsilon}
\def\bd{\textrm{bd}}
\def\int{\textrm{int}}
\def\domain{\mathcal{D}}
\def\diam{\textrm{diam}}
\def\I{\mathcal{I}}
\def\S{\mathcal{S}}
\def\P{\mathcal{P}}
\def\R{\mathbb{R}}
\def\G{\mathcal{G}}
\def\V{\mathcal{V}}
\def\E{\mathcal{E}}
\def\O{\mathcal{O}}

\def\U{\mathcal{U}}

\def\N{\mathbb{N}}
\def\resR{\mathcal{R}}

\def\C{\mathbb{C}}
\def\Ball{\textrm{B}}

\newcommand{\NonW}[2]{\ensuremath{\mathrm{NonW}\left(#1; #2 \right) }}
\newcommand{\Per}[2]{\ensuremath{\mathrm{Per}\left(#1; #2 \right) }}

\newtheorem{theorem}{Theorem}
\newtheorem{lemma}[theorem]{Lemma}
\newtheorem{proposition}[theorem]{Proposition}

\theoremstyle{definition}
\newtheorem{definition}[theorem]{Definition}

\theoremstyle{remark}
\newtheorem{remark}[theorem]{Remark}

\numberwithin{theorem}{section}
\numberwithin{equation}{section}

\providecommand{\norm}[1]{\Vert#1\Vert}

\title{Local stability implies global stability for the 2-dimensional Ricker map}

\author[$\dagger \ddagger$]{Ferenc A. Bartha}
\author[$\ddagger$]{{\'A}bel Garab}
\author[$\ddagger \star$]{Tibor Krisztin\footnote{Corresponding author. Email: krisztin@math.u-szeged.hu}}
\affil[$\dagger$]{\small CAPA group, Department of Mathematics, University of Bergen, Bergen, Norway}
\affil[$\ddagger$]{\small Bolyai Institute, University of Szeged, Szeged, Aradi v. tere 1, 6720, Hungary}
\affil[$\star$]{\small Analyis and Stochastics Research Group of the Hungarian Academy of Sciences,\\ Bolyai Institute, University of Szeged}

\date{}

\begin{document}

\maketitle

\begin{abstract}
\noindent
Consider the difference equation $x_{k+1}=x_k e^{\alpha-x_{n-d}}$ where $\alpha$ is a positive parameter and $d$ is a nonnegative integer. The case $d=0$ was introduced by W.E. Ricker in 1954. For the delayed version $d\ge 1$ of the equation S. Levin and R. May  conjectured in 1976 that local stability of the nontrivial equilibrium implies its global stability. Based on rigorous, computer aided calculations and analytical tools, we prove the conjecture for $d=1$.\\



\end{abstract}

\section{Introduction}
In 1954, Ricker \cite{Ricker_Map} introduced the difference equation
\begin{equation}
  \label{ricker-eq-normal}
  x_{k + 1} = x_k e^{\alpha - x_k}
\end{equation}
with a positive parameter $\alpha$ to model the population density of a single species with non-overlapping generations. The function $R_1:\mathbb{R}\ni x\mapsto xe^{\alpha-x}\in \mathbb{R}$ 
is called the 1-dimensional Ricker map. 
$R_1$ has two fixed points: 0 and $\alpha$. 
It is not difficult to show that $x=\alpha$ is stable if and only if $0<\alpha\le 2$, and, for $0<\alpha\le 2$, $x=\alpha$ attracts all points from $(0,\infty)$; or equivalently, the equilibrium $x=\alpha$ of equation \eqref{ricker-eq-normal} is globally stable provided it is locally stable. 

In 1976, Levin and May \cite{LM} considered the case when there are explicit time lags in the density dependent regulatory mechanisms. This leads to  the difference-delay equation of order $d+1$: 
\begin{equation}
  \label{ricker-eq-delayed}
  x_{k + 1} = x_k e^{\alpha - x_{k-d}},
\end{equation}
where $d$ is a positive integer.

The map 
$$
R_{d+1}: \mathbb{R}^{d+1}\ni (x_0,\ldots,x_d)^T\mapsto 
(x_1,\ldots,x_d,x_de^{\alpha-x_0})^T\in\mathbb{R}^{d+1}
$$
is called the $(d+1)$-dimensional Ricker map; here $T$ denotes the transpose of a vector. 
$R_{d+1}$ has 2 fixed points in $\mathbb{R}^{d+1}$: $(0,\ldots,0)^T$ and $(\alpha,\ldots,\alpha)^T$. 
Levin and May \cite{LM} conjectured in 1976 that local stability of the  
fixed point $(\alpha,\ldots,\alpha)^T\in \mathbb{R}^{d+1}$ implies its global 
stability in the sense that all points from $\mathbb{R}_+^{d+1}:=(0,\infty)^{d+1}$ are attracted by $(\alpha,\ldots,\alpha)^T$. 
As far as we know, the conjecture is still open for all integers $d\ge 1$. 
The aim of the present paper is to prove the conjecture for $d=1$.

Levin and May's conjecture and  many other numerical and analytical studies suggested the folk theorem that `The local stability of the unique positive equilibrium of a single species model implies its global stability.' This claim was recently disproven by a counterexample of Jim{\'e}nez L{\'o}pez \cite{JL, JLP} on global attractivity for Clark's equation \cite{C} when the delay is at least 3.

Liz, Tkachenko and Trofimchuk \cite{LTT} proved that if
\begin{equation}\label{ricker-eq-ltt}
  0<\alpha < \frac{3}{2(d + 1)}
\end{equation}
then the fixed point $(\alpha,\ldots,\alpha)^T\in \mathbb{R}^{d+1}$ of $R_{d+1}$ is globally asymptotically stable, where globally means that the region of attraction of $(\alpha,\ldots,\alpha)^T$ is  $\mathbb{R}_+^{d+1}$. They also suggested that condition \eqref{ricker-eq-ltt} can be replaced by 
\begin{equation}
  \label{ricker-eq-tt}
  0<\alpha < \frac{3}{2(d + 1)} + \frac{1}{2(d + 1)^2},
\end{equation}
which was proven by Tkachenko and Trofimchuk in \cite{TT}. This result is a strong support of the conjecture of Levin and May, and it is proven for a class of maps, not only for $R_{d+1}$. 
For the 1-dimensional Ricker map $R_1$, condition \eqref{ricker-eq-tt} with $d=0$ gives the region $0<\alpha<2$. 
For $d = 1$, i.e., for the 2-dimensional Ricker map $R_2$, condition \eqref{ricker-eq-tt} is equivalent to $0<\alpha<0.875$. See also \cite{L1} and \cite{L2} in the topic.

Linearising $R_2$ at the fixed point $(\alpha,\alpha)^T$ shows that local 
exponential stability of $(\alpha,\alpha)^T$ holds for $0<\alpha<1$, and 
  $(\alpha,\alpha)^T$ is unstable for $\alpha>1$.
As $\alpha$ passes the value 1, a Neimark--Sacker bifurcation takes place 
at the nontrivial fixed point.  
In this paper we show that global asymptotic stability is true also for the  
parameter values  $\alpha \in [0.875, 1]$. We emphasise that 
our result implies global stability at the critical parameter value 
$\alpha=1$ as well.

In case $d = 1$ the difference equation (\ref{ricker-eq-delayed}) is equivalent to the 2-dimensional system
\begin{equation}
  \label{ricker-eq-main}
\begin{split}
  x_{k + 1} &= y_k, \\
  y_{k + 1} &= y_k e^{\alpha - x_k},
\end{split}
\end{equation}
and this is also equivalent to the 2-dimensional discrete dynamical system generated by the 2-dimensi\-onal Ricker map $R_2$. 
As $d=1$ will be fixed in the remaining part of the paper, we shall use the notation $F_\alpha$ instead of $R_2=R_2(\alpha)$. 
$F_\alpha$ has two fixed points $(0,0)^T$ and $(\alpha,\alpha)^T$. 
From now on, we shall analyse the dynamics generated by $F_\alpha$ in the positive quadrant $\R^2_+=(0,\infty)\times(0,\infty)$. Note that $F_\alpha(\overline{\R^2_+})\subseteq \overline{\R^2_+}$ and $F_\alpha(\R^2_+)\subset \R^2_+$. 

We shall use a combination of analytic and computational arguments. The latter is done using interval arithmetics, that is a standard in the area of validated or rigorous numerics. Instead of calculating with numbers, we use intervals to control the errors introduced by the computer. After finishing a computation, the information we obtain is that the true result is contained in the result interval. We shall draw conclusions from that. An interval $[a]$ is represented as a pair of endpoints $[a^-, a^+]$. Having a set $S$ or a number $r$, we denote their interval enclosures by $[S]$ and $[r]$, respectively. The reader is referred to Moore \cite{Moore_Methods}, Alefeld \cite{Alefeld_Introduction}, Tucker \cite{Tucker_Lorenz}, \cite{Tucker_Validated}, Nedialkov, Jackson and Corliss \cite{Nedialkov_Validated} for further details.

The structure of the paper is as follows. 
In Section \ref{invariant} we construct a compact region $S$, which is a closed square around  $(\alpha,\alpha)^T$, having the property that $F_\alpha(S)\subseteq S$ and every trajectory enters it eventually.
In Section \ref{neighbourhood} we construct an attracting neighbourhood of the fixed point $(\alpha,\alpha)^T$. We use two different approaches. First, 
  the linear approximation of $F_\alpha$ is applied at the fixed point. Naturally, the size of this neighbourhood tends to $0$ as the parameter $\alpha$ tends to the critical value 1. 
Then we analyse the normal form of $F_\alpha$ used in the study of the Neimark--Sacker bifurcation, and obtain a uniform neighbourhood in the parameter range $\alpha\in[0.999,1]$ belonging to the basin of attraction of the nontrivial fixed point $(\alpha,\alpha)^T$. 
In addition to the standard techniques applied in the Neimark--Sacker bifurcation, we need explicit estimations on the sizes of the higher order terms in order to get a sufficiently large attracting neighbourhood of the fixed point  $(\alpha,\alpha)^T$. 
In Section \ref{graph} we give an overview on graph representations of discrete dynamical systems and show how it can be used to study qualitative properties of dynamical systems. 
For possible future applications we formulate two approaches for general 
continuous maps in Euclidean spaces. 
In particular, the correctness of an algorithm is verified in order to enclose 
non-wandering points.    
In Section \ref{application} we combine the computational techniques of Section \ref{graph}  and rigorously show that every trajectory of $F_\alpha$ starting from $\mathbb{R}_+^2$ enters the neighbourhood constructed 
in Section  \ref{neighbourhood}. 
This proves our main result:

\begin{theorem}\label{maintheorem}
If $0<\alpha\leq 1$, then $(\alpha,\alpha)^T$ is locally stable,  and $F_{\alpha}^n(x,y)\to(\alpha,\alpha)^T$ as $n \to \infty$ for all $(x,y)^T\in \R^2_+$.
\end{theorem}

For the sake of completeness, here we give a proof for all $\alpha \in (0,1]$. The result is new only for $\alpha\in[0.875,1]$.

There is an appendix containing some large formulae used in Section \ref{neighbourhood}.  The program codes and results of our rigorous computer aided computations can be found on link \cite{web}.

Consideration of a bifurcation of a given dynamical system is usually used to show that some phenomenon appears in the global dynamics of the system as a parameter passes a critical value. The invention in our method is that we use the normal form of a bifurcation in combination with the tools of graph representations of dynamical systems and interval arithmetics to prove the absence of a phenomenon for certain parameter values near the critical one. 
As we want to construct explicitly given and computationally useful regions, 
the key technical difficulty is the estimation of the sizes of the higher order (error) terms in the normal forms.  
We hope that our proof shows that these ideas are applicable for a wide range of  discrete or continuous dynamical systems, as well.

Running the program of D\'enes and Makay \cite{DM}, which is developed to (nonrigorously) find and visualise attractors and basins of discrete dynamical systems, suggests that the conjecture of Levin and May stands for the 3-dimensional Ricker model, as well. 
In order to prove the conjecture in this case and also for larger values of $d$, an additional technical difficulty arises. Namely, first a center manifold reduction is necessary, and the construction of an attracting neighbourhood 
should be done on the center manifold. Among others, an explicit estimation of the size 
of the center manifold will play a crucial role as well.  

\subsection*{Notations and definitions}
Throughout the paper some further notations and definitions will be used. 
$\mathbb{N}$, $\mathbb{N}_0$, $\mathbb{R}$, $\mathbb{C}$ stand for the set of 
positive integers, non-negative integers, real numbers, and complex numbers, respectively.  The open ball in the Euclidean-norm $\norm{.}$  and in the maximum norm with radius $\delta \geq 0$ around $q \in \R^n$ are denoted by $\Ball(q; \delta)$ and $\mathrm{K}(q;\delta)$, respectively.  In Section \ref{neighbourhood}  we shall often use the notation\linebreak $\mathrm{B}_\delta=\{z\in \C: \ |z|<\delta\}$ for $\delta>0$, where $|z|$ denotes the absolute value of $z\in\C$. For a vector $x\in \C^n$, $x^T$ denotes the transpose of $x$.  For $\xi=(\xi_1,\xi_2)^T\in \C^2$ and $\zeta=(\zeta_1,\zeta_2)^T\in\C^2$ let $\langle \xi,\zeta \rangle$ denote the scalar product of them defined by  $\langle \xi,\zeta\rangle=\overline{\xi_1}\zeta_1+\overline{\xi_2}\zeta_2$. Let also $\underline{\alpha}=(\alpha,\alpha)^T$.

Let $f:\domain_f\subseteq\mathbb{R}^n\to\mathbb{R}^n$ be a continuous map. 
For $k\in \mathbb{N}_0$, $f^k$ denotes the $k$-fold composition of $f$, i.e., $f^{k+1}(x)=f(f^k(x))$, and $f^0(x)=x$. A fixed point $p \in \domain_f$ 
of $f$ is called locally stable if for every $\varepsilon>0$ 
 there exists $\delta>0$ such that $\norm{x-p}<\delta$ implies $\norm{f^n(x)-p}<\varepsilon$ for all $n\in \mathbb{N}$. We say that the fixed point $p$ attracts the region $U\subseteq\domain_f$ if for all points $u\in U$,  $\norm{f^n(u)- p}\to 0$ as $n\to \infty$. The fixed point $p$ is globally attracting if it attracts all of $\domain_f$ and is globally stable if it is locally stable and globally attracting.

\section{The dynamics in the first quadrant}
\label{invariant}
In this section we construct compact squares $S_i^{(\alpha)}\subset\R^2_+, \ i\in\N_0$ around $\underline{\alpha}$ so that $F_\alpha(S_i^{(\alpha)})\subseteq S_i^{(\alpha)}$ and $S_i^{(\alpha)}$ attracts all points of $\R^2_+$  for all $i\in\N_0$ and $\alpha\in(0,1]$. Hence an elementary proof of Theorem \ref{maintheorem} is obtained for $0 < \alpha \leq 0.5$.
Recall $F_\alpha(\R^2_+) \subseteq \R^2_+$. We can illustrate the image $(x_{k+1}, y_{k+1})$ of $(x_k, y_k)$ as first going horizontally from $(x_k, y_k)$ to the diagonal, proceeding upwards if $0 < x_k < \alpha$, otherwise downwards vertically until we reach the value $y_{k+1} = y_k e^{\alpha - x_k}$.
This is shown on Figure \ref{ricker-fig-trapping}.
\begin{figure}[H]
  \centering
  \includegraphics[scale=0.9]{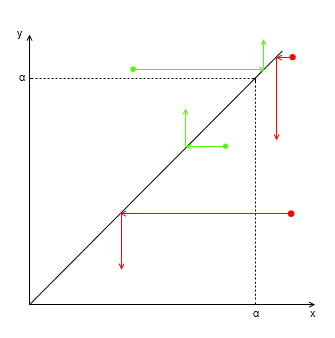}
  \caption{\label{ricker-fig-trapping}Dynamics for $x > 0, y > 0$}
\end{figure}
\noindent
Let $(x_0, y_0) \in \R^2_+$ and $0 < \alpha \leq 1$ be given. Define the sequence $(x_k, y_k)^\infty_{k = 0} \in \R^2_+$ by
\begin{equation*}
(x_{k+1}, y_{k+1}) = F_\alpha(x_k, y_k), \; k \in \{ 0, 1, \ldots \}. 
\end{equation*}
Consider the following cases:
\begin{enumerate}[I.]

  \item $0 < x_0 \leq \alpha \leq y_0$

  Clearly we have $\alpha \leq x_1 \leq y_1$ and $\max \{x_0, y_0, x_1, y_1\} \leq y_1$.

 \item $\alpha \leq x_0 \leq y_0$

  In this case $\alpha \leq x_1$ and $y_1 \leq x_1$ with $\max \{x_0, y_0, x_1, y_1\} \leq y_0$. We distinguish two cases depending on $y_1 \leq \alpha$ or not.

 \item $\alpha \leq y_0 \leq x_0$

  We obtain $\alpha \leq y_0 = x_1$ and $y_1 \leq y_0 = x_1$. During the consequent iterations $y_{i+1} \leq y_i = x_{i+1}$ is satisfied as long as $\alpha \leq x_i$ stays true. If $\alpha \leq x_i$ for all $i$, then $y_i > 0$ for all $i$, and both $(y_i)^\infty_{i = 0}$ and $(x_i)^\infty_{i = 0}$ are monotonically decreasing, and converge. The only possibility is $(x_k, y_k) \to (\alpha, \alpha)$, since the other fixed point is at $(0,0)$. Otherwise there is a minimal $N > 0$ such that $0 < y_N \leq x_N < \alpha$ is satisfied. We note that $0 < y_{N-1} < \alpha \leq x_{N-1}$ is true. We have $\max_{i \in \{0, \ldots, N\}} \{x_i, y_i\} < x_0$.

  \item $0 < y_0 < \alpha \leq x_0$

  Obviously $0 < y_1 \leq x_1 < \alpha$, and $\max \{x_0, y_0, x_1, y_1\} \leq x_0$.

  \item $0 < y_0 \leq x_0 \leq \alpha$

  Here we have $0 < x_1 \leq \alpha$, $0 < x_1 \leq y_1$ and $\max \{ x_0, y_0, x_1, y_1 \} \leq y_1$. We distinguish two cases depending on $\alpha \leq y_1$ or not.

  \item $0 < x_0 \leq y_0 \leq \alpha$

  Then $x_1 \leq \alpha$ and $x_1 \leq y_1$. Now $x_{i+1} = y_i \leq y_{i+1}$ is satisfied as long as $x_i \leq \alpha$ stays true. If $x_i \leq \alpha$ for all $i$, then $y_i > 0$ for all $i$, and both $(y_i)^\infty_{i = 0}$ and $(x_i)^\infty_{i = 0}$ are monotonically increasing, and converge. The only possibility is $(x_k, y_k) \to (\alpha, \alpha)$, since the other fixed point is at $(0,0)$. Otherwise there is a minimal $N > 0$ such that $\alpha < x_N \leq y_N$ is satisfied. We note that $0 < x_{N - 1} \leq \alpha < y_{N-1}$. We have $\max_{i \in \{0, \ldots, N\}} \{x_i, y_i\} < y_N$. 
\end{enumerate}
We conclude that for any $(x_0, y_0) \in \R^2_+$ the sequence $\left((x_k, y_k)\right)^\infty_{k = 0}$ converges to $(\alpha, \alpha)$ or enters the triangle $H(h_0^{(\alpha)}) = \{ (x, y) : h_0^{(\alpha)} < y \leq x \leq \alpha \}$, with $h_0^{(\alpha)} = 0$, that is the region denoted by $V$ in Figure \ref{ricker-fig-rotation}. We mark the possible transitions with arrows, if it is solid, it refers to transition in one step, otherwise it is possible to have multiple iterations before entering the next region.
\begin{figure}[H]
  \centering
  \includegraphics[scale=0.7]{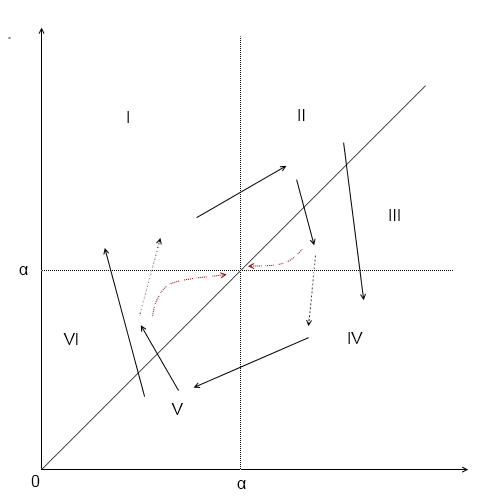}
  \caption{\label{ricker-fig-rotation}Dynamics for $x > 0, y > 0$}
\end{figure}
Assume now that $(x_0, y_0) \in H(h_0^{(\alpha)})$. Either the sequence converges to $(\alpha, \alpha)$ or there exists an $N > 1$ such that $(x_{N - 1}, y_{N - 1}) \in V \cup VI$, $(x_N, y_N) \in I$ and $(x_{N + 1}, y_{N + 1}) \in II$.
So the following stands:
\begin{equation*}
  y_{N + 1} = y_N e^{\alpha - x_N} = y_{N - 1} e ^ {2 \alpha - x_N - x_{N - 1}} \leq \alpha e^{2 \alpha - 2h_0^{(\alpha)}}.
\end{equation*}
This implies that $(x_{N + 1}, y_{N + 1}) \in G(g_0^{(\alpha)})$, where $G(g_0^{(\alpha)}) = \{ (x, y) : \alpha \leq x \leq y \leq g_0^{(\alpha)} \}$, with $g_0^{(\alpha)} = \alpha e^{2 \alpha - 2h_0^{(\alpha)}}$. Now there exists $M \geq N + 2$ such that $(x_{M - 1}, y_{M - 1}) \in II \cup III$, $(x_M, y_M) \in IV$ and $(x_{M + 1}, y_{M + 1}) \in V$. We get the following inequality:
\begin{equation*}
  y_{M + 1} = y_{M - 1} e ^ {2 \alpha - x_M - x_{M - 1}} \geq \alpha e^{2 \alpha - 2g_0^{(\alpha)}}.
\end{equation*}
Therefore $(x_{M + 1}, y_{M + 1}) \in H(h_1^{(\alpha)})$, with $h_1^{(\alpha)} = \alpha e^{2 \alpha - 2g_0^{(\alpha)}}$. Similarly, the sequence will enter the set $G(g_1^{(\alpha)})$, with $g_1^{(\alpha)} = \alpha e^{2 \alpha - 2h_1^{(\alpha)}}$. Repeating this argument we get two sequences $(h_i^{(\alpha)})^\infty_{i=0}$ and $(g_i^{(\alpha)})^\infty_{i=0}$ defined recursively by $h_0^{(\alpha)} = 0$, $h_i^{(\alpha)} = \alpha e^{2 \alpha - 2g_{i-1}^{(\alpha)}}$ for $i \geq 1$, and $g_i = \alpha e^{2 \alpha - 2h_i^{(\alpha)}}$ for $i \geq 0$. It is easy to see that $(h_i^{(\alpha)})$ increases and $(g_i^{(\alpha)})$ decreases. We have the limits $h_i^{(\alpha)} \nearrow h_\infty^{(\alpha)} \leq \alpha$ and $g_i^{(\alpha)} \searrow g_\infty^{(\alpha)} \geq \alpha$. Define
\begin{equation}
  \label{ricker-eq-trapping}
  S_i^{(\alpha)} = \{ (x, y) : h_i^{(\alpha)} \leq x \leq g_i^{(\alpha)}, h_i^{(\alpha)} \leq y \leq g_i^{(\alpha)} \}.
\end{equation}
We sum our observations in the following Lemma:
\begin{lemma}
  \label{ricker-lemma-trapping}
  For every $\alpha \in (0,1]$ and for every $i \in \N_0$
  \begin{enumerate}
   \item $F_\alpha(S_i^{(\alpha)}) \subseteq S_i^{(\alpha)}$,
   \item $S_i^{(\alpha)}$ attracts all points of $\R^2_+$.
  \end{enumerate}
\end{lemma}

If $S_\infty^{(\alpha)} = \{(\alpha, \alpha)\}$, then we have already established that the fixed point is globally attracting. This is not the case however for $\alpha > 0.5$. For a given $\alpha$, we can view the sequences $(h_i^{(\alpha)})$ and $(g_i^{(\alpha)})$ as even and odd iterates of the function
\begin{equation}
  \label{ricker-eq-trapping-func}
  \tau_\alpha(t) = \alpha e^{2 \alpha - 2 t},
\end{equation}
starting from $t = 0$. That is, $h_0^{(\alpha)} = \tau_\alpha^0(0), \; g_0^{(\alpha)} = \tau_\alpha^1(0), \; h_1^{(\alpha)} = \tau_\alpha^2(0), \; g_1^{(\alpha)} = \tau_\alpha^3(0), \; \ldots \; .$ The non-rigorous bifurcation diagram on Figure \ref{ricker-fig-K-iteration} shows that the unique fixed point $\alpha$ of $\tau_\alpha$ is stable for $0 < \alpha < 0.5$, it is unstable for $\alpha > 0.5$, and there is an attracting 2-cycle for $\alpha > 0.5$. Thus $S_\infty^{(\alpha)} = \{(\alpha, \alpha)\}$ can be expected only for $\alpha \in (0,0.5)$.  For $\alpha > 0.5$ the two points of the 2-cycle give $h_\infty^{(\alpha)}$ and $g_\infty^{(\alpha)}$.
\begin{figure}[H]
  \centering
  \includegraphics[scale=0.7]{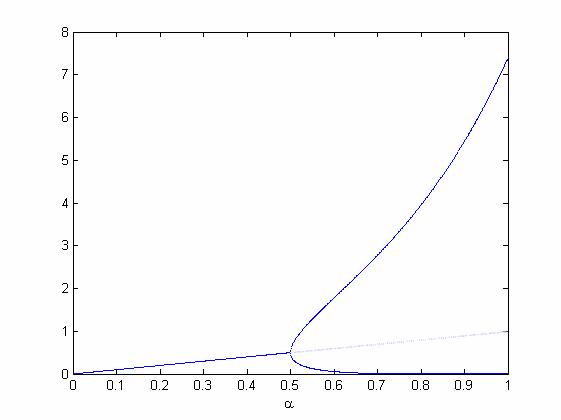}
  \caption{\label{ricker-fig-K-iteration}Bifurcation diagram of $\tau_\alpha$}
\end{figure}
\noindent
\begin{proposition}
  \label{ricker-thm-trapping}
  If $0 < \alpha \leq 0.5$, then for every $(x,y)\in \R^2_+$ we have $F_{\alpha}^n(x,y) \to (\alpha,\alpha)$ as $n \to \infty$.
\end{proposition}
\begin{proof}
Observe that $\tau_\alpha'(t) = - 2 \tau_\alpha(t)$, $\tau_\alpha(t) > 0$, and if $t > \alpha$ and $\alpha \in (0, 0.5]$, then $2 \tau_\alpha(t) = 2 \alpha \leq 1$ is satisfied. Since
\begin{equation}
  \frac{d^2}{d t^2} \left( \tau_\alpha(\tau_\alpha(t)) - t \right) = 8 \tau_\alpha(\tau_\alpha(t)) \tau_\alpha(t) \left( 2 \tau_\alpha(t) - 1 \right),
\end{equation}
therefore $\tau_\alpha(\tau_\alpha(t)) - t$ is a concave function for $t > \alpha$. The first derivative at $\alpha \in (0, 0.5]$ is
\begin{equation}
  \left. \frac{d}{d t} \left( \tau_\alpha(\tau_\alpha(t)) - t \right) \right|_{t = \alpha} = 4 \tau_\alpha(\tau_\alpha(\alpha)) \tau_\alpha(\alpha) - 1 = 4 \alpha^2 - 1 \leq 0.
\end{equation}
These imply that $\frac{d}{d t} \left( \tau_\alpha(\tau_\alpha(t)) - t \right) < 0$ for all $t > \alpha$. Thus the only zero of $\tau_\alpha(\tau_\alpha(t)) - t$ in $t \in [\alpha, \infty)$ is $\alpha$. This gives us that $g_\infty^{(\alpha)} = \alpha$ which implies that $h_\infty^{(\alpha)} = \alpha$ and consequently $S_\infty^{(\alpha)} = \{(\alpha, \alpha)\}$ in this parameter region.
\end{proof}
\noindent
We assume $\alpha \in [0.5, 1]$ in the remaining part of the paper.

\section{Attracting neighbourhood}\label{neighbourhood}
Let us consider map $F_\alpha$. In this section we are going to give $\eps(\alpha)>0$, such that \linebreak
$\inf_{\alpha\in[0.5,1]}\eps(\alpha)>0$ and $\mathrm{K}(\underline{\alpha};\eps(\alpha))$ belongs to the basin of attraction of $\underline{\alpha}$ for $\alpha\in[0.5,1]$, that is, $F_\alpha^n(x_0,y_0)\to \underline{\alpha}$ as $n\to\infty$ for all $(x_0,y_0)^T\in \mathrm{K}(\underline{\alpha};\eps(\alpha))$.

Introducing the new variables $u=x-\alpha, \ v=y-\alpha$, map $F_\alpha$ can be written in the form
\begin{equation}\label{Ricker_beta_F}
\left(\begin{array}{c}
u\\
v  \end{array}\right)
\mapsto A(\alpha)\left(\begin{array}{ccc}
u\\
v\end{array}\right) + f_\alpha(u,v),
\end{equation}
where the linear part is
$$
A(\alpha)=\left(\begin{array}{ccc}
0 &1\\
-\alpha &1
\end{array}\right)$$
and the remainder is
$$
f_\alpha(u,v)=\left(\begin{array}{ccc}
0\\
v(e^{-u}-1)+\alpha(e^{-u}-1+u) \end{array}\right)
.$$
The eigenvalues of $A(\alpha)$ are $\mu_{1,2}(\alpha)=\frac{1\pm i\sqrt{4\alpha-1}}{2}\in\C$, and the corresponding complex eigenvectors are $q_{1,2}(\alpha)=
\left(\frac{1\mp\sqrt{1-4\alpha}}{2\alpha},1\right)^T=\left(\frac{1\mp i\sqrt{4\alpha-1}}{2\alpha},1\right)^T\in\C^2,$ respectively for $\alpha>\frac{1}{4}$. Let $q(\alpha)=q_1(\alpha)$ and $\mu(\alpha)=\mu_1(\alpha)$. Let $p(\alpha) \in \C^2$ denote the eigenvector of $A(\alpha)^T$ corresponding to $\overline{\mu(\alpha)}$ such that  $\langle p(\alpha),q(\alpha)\rangle=1,$. This leads to
$$p(\alpha)=\left(-\frac{i\alpha}{\sqrt{4\alpha-1}},\frac{\sqrt{4\alpha-1}+i}{2\sqrt{4\alpha-1}}\right)^T.$$
We introduce a complex variable
\begin{equation}\label{z(u,v)}
z=z(u,v,\alpha)=\langle p(\alpha),(u,v)^T\rangle=\frac{1}{2}\left(v-\frac{i (v-2 u \alpha )}{\sqrt{-1+4 \alpha }}\right).
\end{equation}
We also have an explicit formula for $(u,v)^T$ in terms of $z$, which reads as
\begin{equation}\label{uv(z)}
(u,v)^T=zq(\alpha)+\overline{zq(\alpha)}=
\left(\frac{\mathrm{Re}z+\sqrt{4\alpha-1}\mathrm{Im}z}{\alpha},2\mathrm{Re}z  \right)^T.
\end{equation}
Our original system \eqref{Ricker_beta_F} is now transformed into the complex system
\begin{equation}\label{Ricker_complex_z}
\begin{split}
z\mapsto G(z,\overline{z},\alpha)&=\left\langle p(\alpha),A(\alpha)(zq(\alpha)+\overline{zq(\alpha)})
+f_\alpha(zq(\alpha)+\overline{zq(\alpha)})\right\rangle\\
&=\mu(\alpha)z+g(z,\overline{z},\alpha),
\end{split}
\end{equation}
where $g$ is a complex valued smooth function of $z,\ \overline{z}$ and $\alpha$, defined by \eqref{app_g}. It can also be seen that for fixed $\alpha$, $g$ is an analytic function of $z$ and $\overline{z}$ and the Taylor expansion of $g$ with respect to $z$ and $\overline{z}$ has only quadratic and higher order terms. That is,
$$g(z,\overline{z},\alpha)=\sum_{2\leq k+l} \frac{1}{k!l!}g_{kl}(\alpha)z^k\overline{z}^l, \quad \text{with } k,l=0,1,\dots,$$
where $g_{kl}(\alpha)=\left.\frac{\partial^{k+l}}{\partial z^k \partial \overline{z}^l}g(z,\overline{z},\alpha)\right|_{z=0}$ for $k+l\geq 2,\ k,l=\{0,1,\dots\}$.

\begin{proposition}\label{thm_epsalpha}
Let  $\alpha \in [0.5,1)$ be fixed and
$$\varepsilon(\alpha)=\min\left\{\frac{1}{20}\sqrt{\frac{4\alpha-1}{2+\alpha}},
\frac{9(4\alpha-1)(1-\sqrt{\alpha})}{20(1+2\sqrt{\alpha})\sqrt{2+\alpha}}\right\}.$$
Then $\left\{(x,y)^T\in\R^2:|x-\alpha|<\eps(\alpha), |y-\alpha|<\eps(\alpha)\right\}$ belongs to the basin of attraction of the fixed point  $\underline{\alpha}$ of $F_\alpha$.
\end{proposition}
\begin{proof}
Let us study the map in the form \eqref{Ricker_complex_z}. First note that \eqref{uv(z)} easily implies inequalities
\begin{equation}\label{u,v<z}
|u|\leq \frac{2}{\sqrt{\alpha}}|z| \text{ and } |v|\leq 2|z|.
\end{equation}
Assuming $|u|<1/10$ and $|v|<1/10$ and applying the inequalities $\left|e^{-u}-1\right|\leq e^{1/10}|u|\leq \frac{10}{9}|u|$ and $\left|e^{-u}-1+u\right|\leq e^{1/10}\frac{u^2}{2}\leq \frac{5}{9}u^2$  we obtain the following estimations:
\begin{equation*}
\begin{split}
 |g(z,\overline{z},\alpha)|&=\left|\left\langle p(\alpha),f_\alpha\left(zq(\alpha)+\overline{zq(\alpha)}\right)\right\rangle\right|\\
&=\left|\frac{\sqrt{4\alpha-1}+i}{2\sqrt{4\alpha-1}}\right|
\left|v(e^{-u}-1)+\alpha(e^{-u}-1+u)\right|\\
&\leq \sqrt{\frac{\alpha}{4\alpha-1}}\left(|v||e^{-u}-1|+\alpha|e^{-u}-1+u|\right)\leq 
\sqrt{\frac{\alpha}{4\alpha-1}}\left(|uv|e^{1/10}+\alpha\frac{e^{1/10}}{2}u^2\right)\\
&\leq\frac{5}{9}\sqrt{\frac{\alpha}{4\alpha-1}}\left(u^2+2|uv|)\right)
\leq \frac{5}{9}\cdot \frac{4(1+2\sqrt{\alpha})}{\sqrt{\alpha(4\alpha-1)}}|z|^2.
\end{split}
\end{equation*}
Now, since $|\mu(\alpha)|=\sqrt{\alpha}$, hence
$$|G(z,\overline{z},\alpha)|\leq \left(\sqrt{\alpha}+\frac{5}{9}\cdot \frac{4(1+2\sqrt{\alpha})}{\sqrt{\alpha(4\alpha-1)}}|z|\right)|z|<|z|$$
provided that $|z|\neq 0$ is so small that $|u|<\frac{1}{10}$ and $|v|<\frac{1}{10}$ and $\sqrt{\alpha}+\frac{5}{9}\cdot \frac{4(1+2\sqrt{\alpha})}{\sqrt{\alpha(4\alpha-1)}}|z|<1$.
Inequalities \eqref{u,v<z} imply that $|z|<\frac{\sqrt{\alpha}}{20}$ guarantees $|u|<\frac{1}{10}$ and $|v|<\frac{1}{10}$. Therefore
\begin{equation}\label{lineareps_z}
0<|z|<\eps_G(\alpha)=\min\left\{\frac{\sqrt{\alpha}}{20},
\frac{9(1-\sqrt{\alpha})\sqrt{\alpha(\alpha-1)}}{20(1+2\sqrt{\alpha})}\right\}
\end{equation}
implies $|G(z,\overline{z},\alpha)|<|z|$, which means that $|G^n(z_0,\overline{z_0},\alpha)|\to 0$ as $n\to \infty$ if $|z_0|<\eps_G(\alpha)$ is satisfied. We show this by way of contradiction. Assume that $|z_0|<\eps_G(\alpha)$, $z_n=G^n(z_0,\overline{z_0},\alpha)$ and $|z_0|>|z_1|>\dots >|z_n|>\dots\geq 0$ with $|z_n|\to c>0$ as $n\to \infty$. Since $G$ is continuous we have that $\max_{|z|=c}|G(z,\overline{z},\alpha)|<c$ and consequently $|z_k|<c$ also holds if  $k$ is large enough which is a contradiction.

From equation \eqref{z(u,v)} one obtains that if $|u|<\delta,\ |v|<\delta$, then
\begin{equation}\label{z<deltaalpha}
|z|<\delta\sqrt{\frac{\alpha(2+\alpha)}{4\alpha-1}}.
\end{equation}
From \eqref{z<deltaalpha} we infer that if $|u|<\eps(\alpha)$ and $|v|<\eps(\alpha)$ then $|z(u,v,\alpha)|<\eps_G(\alpha)$ which completes our proof.
\end{proof}

Note that $\varepsilon(\alpha)$ goes to $0$ as $\alpha$ goes to $1$. This means that when $\alpha$ is close to $1$, then the constructed region $\mathrm{K}(\underline{\alpha};\eps(\alpha))$ becomes very small. Thus it is impossible to show by interval arithmetic tools that every trajectory enters into it eventually. Nevertheless $\eps(\alpha)$ might be used in the case $\alpha\in[0.5,0.999]$. However, our following method is not only capable to give an attracting neighbourhood for all $\alpha \in [0.999,1]$, whose size is independent of the concrete value of the parameter, but it also serves a better estimation of the attracting region even if we assume only $\alpha\in[0.875,1]$. The main results of the section are the following two propositions.  Here, we only present the proof of Proposition \ref{thm22}. The whole argument can be repeated to get a universal attracting neighbourhood when only $[0.875,1]$ is assumed. The differences only appear in concrete values in the given estimations. 

\begin{proposition}\label{thm37}
For all fixed $\alpha \in [0.875,1]$, $\left\{(x,y)^T\in\R^2:|x-\alpha|<1/37, |y-\alpha|<1/37\right\}$ belongs to the basin of attraction of the fixed point  $\underline{\alpha}$ of $F_\alpha$.
\end{proposition}

\begin{proposition}\label{thm22}
For all fixed $\alpha \in [0.999,1]$, $\left\{(x,y)^T\in\R^2:|x-\alpha|<1/22, |y-\alpha|<1/22\right\}$ belongs to the basin of attraction of the fixed point  $\underline{\alpha}$ of $F_\alpha$.
\end{proposition}

\begin{proof}
We follow the steps of finding the normal form of the Neimark--Sacker bifurcation, according to Kuznetsov \cite{Kuznetsov}. In our calculations and estimations we use symbolic calculations and built in symbolic interval arithmetic tools of Wolfram Mathematica v. 7 or 8.

According to Kuznetsov \cite{Kuznetsov}, we are aiming to transform system \eqref{Ricker_complex_z} to a system which takes the following form.
\begin{equation}\label{Ricker_complex_w_bev}
w\mapsto \mu(\alpha)w+c_1(\alpha)w^2\overline{w}+R_2(w,\overline{w},\alpha),
\end{equation}
where $c_1$ and $R_2$ are smooth, real functions such that for fixed $\alpha$,  $R_2(w,\overline{w},\alpha)=O(|w|^4)$. We are going to show that there exists $\eps_0>0$ such that for all $|w|<\eps_0$ and $\alpha \in [0.999,1]$, $$\left|\mu(\alpha)w+c_1(\alpha)w^2\overline{w}+R_2(w,\overline{w},\alpha)\right|<|w|$$
holds, which implies that $\mathrm{B}_{\eps_0}$ belongs to the basin of attraction of the fixed point $0$ of the discrete dynamical system generated by \eqref{Ricker_complex_w_bev}. From this, we shall be able to show that the fixed point $\underline{\alpha}$ of $F_\alpha$ attracts all points of $\mathrm{K}(\underline{\alpha};\frac{1}{22})$.

To be more precise, for a fixed $\alpha$, we are looking for a function $h_{\alpha}:\C\to\C$, which is invertible in a neighbourhood of $0\in\C$ and which is such that in the new coordinate $w=h_{\alpha}^{-1}(z)$, our map \eqref{Ricker_complex_z} takes the form 
\begin{equation}\label{Ricker_complex_w}
w\mapsto h_{\alpha}^{-1}(G(h_{\alpha}(w),\overline{h_{\alpha}(w)},\alpha))=\mu(\alpha)w+c_1(\alpha)w^2\overline{w}+R_2(w,\overline{w},\alpha),
\end{equation}
where $R_2(w,\overline{w},\alpha)=O(|w|^4)$ for fixed $\alpha$. According to \cite{Kuznetsov}, such a function can be found in the form
\begin{equation}\label{h(w)}
h_{\alpha}(w)=w+\frac{h_{20}(\alpha)}{2} w^2+h_{11}(\alpha) w \overline{w}+\frac{h_{02}(\alpha)}{2}\overline{w}^2+\frac{h_{30}(\alpha)}{6}w^3+\frac{h_{12}(\alpha)}{2} w \overline{w}^2+\frac{h_{03}(\alpha)}{6}\overline{w}^3.
\end{equation}

Clearly, $h_\alpha$ has an inverse in a small neighbourhood of $0\in\C$. A formula for  $h_\alpha^{-1}$ can be given in the form
$$h^{-1}_\alpha(z)=h_{inv,\alpha}(z)+R_3(z,\alpha),$$
where 
$$h_{inv,\alpha}(z)=z+\sum_{2\leq k+l \leq 4} a_{kl}(\alpha)z^k\overline{z}^l$$
and $R_3(z,\alpha)=O(|z|^5)$. The coefficients can be obtained by substituting $w=h^{-1}(z)$ into $z=h_\alpha(w)$ and equating the coefficients of the same type of terms up to the fourth order. The result for $h_{inv,\alpha}$ in terms of $h_{20}(\alpha),\dots,h_{03}(\alpha)$ is given in \eqref{app_hinv(z)Hij}. The coefficients $h_{20}(\alpha),\dots,h_{03}(\alpha)$ are determined such that
$$h_{\alpha}^{-1}(G(h_\alpha(w),\overline{h_\alpha(w)},\alpha))$$
has the form $\mu(\alpha)+c_1(\alpha)w^2\overline{w}$ plus at least fourth order terms in $w$, that is, the transformation kills all second and third order terms with one exception. This requires the condition
$$\left(\frac{\mu(1)}{|\mu(1)|}\right)^k\neq 1 \quad \text{for all}\ k\in\{1,2,3,4\}.$$
As $\mu(1)=\frac{1}{2}+i\frac{\sqrt{3}}{2}$, this is clearly satisfied. Formulae \eqref{app_h20}-\eqref{app_h03} contain the obtained results.

In which region is the transformation valid? We are going to show that $h_\alpha$ is injective on $\overline{\mathrm{B}_{1/3}} \subset \C$ and that $h^{-1}_\alpha$ is defined on $\overline{\mathrm{B}_{1/5}}$. Let us suppose that $z\in\C$ is fixed and $h_{20}(\alpha),\dots,h_{03}(\alpha)$ are given for a fixed $\alpha \in [0.999,1]$. Let 
$$H_{\alpha,z}:\C\ni w\mapsto w+z-h_\alpha(w)\in \C.$$
By this notation, $H_{\alpha,z}(w)=w$ holds if and only if $h_\alpha(w)=z$. Now, we have the following
\begin{equation*}
\begin{split}
\left|H_{\alpha,z}(w_1)-H_{\alpha,z}(w_2)\right|=&\left|w_1-h_\alpha(w_1)-w_2+h_\alpha(w_2)\right|\\
\leq &|w_1-w_2|\cdot\left(\left(\frac{|h_{20}(\alpha)|}{2}+|h_{11}(\alpha)| +\frac{|h_{02}(\alpha)|}{2}\right)(|w_1|+|w_2|)\right.+\\
&\left.\left(\frac{|h_{30}(\alpha)|}{6}+\frac{|h_{12}(\alpha)|}{2}+
\frac{|h_{03}(\alpha)|}{6}\right)\left(|w_1|^2+|w_1||w_2|+|w_2|^2\right)\right).
\end{split}
\end{equation*}
If $|h_{20}(\alpha)|/2+|h_{11}(\alpha)|+|h_{02}(\alpha)|/2<\delta_1,\ |h_{30}(\alpha)|/6 + |h_{12}(\alpha)|/2 + |h_{03}(\alpha)|/6<\delta_2,\ |w|\leq \delta_3$ and $|z|\leq \delta_4$ hold, then we have
$$\left|H_{\alpha,z}(w_1)-H_{\alpha,z}(w_2)\right|\leq |w_1-w_2|(2\delta_1\delta_3+3\delta_1\delta_3^2)$$
and
$$\left|H_{\alpha,z}(w)\right|\leq \delta_4+\delta_1\delta_3^2+\delta_2\delta_3^3.$$
By interval arithmetics we obtain that the first two inequalities are fulfilled if $\delta_1=0.76$ and $\delta_2=0.52$. Now, if we choose $\delta_3=\frac{1}{3}$ and $\delta_4=\frac{1}{5}$ we obtain that $H_{\alpha,z}:\overline{\mathrm{B}_{1/3}}\to \overline{\mathrm{B}_{1/3}}$ is a contraction. Hence for all fixed $z \in \overline{\mathrm{B}_{1/5}}$ there exists exactly one $w=w(z)\in\overline{\mathrm{B}_{1/3}}$ such that $H_{\alpha,z}(w(z))=w(z)$, that is $h_\alpha(w(z))=z$. This means that $h_\alpha^{-1}$ can be defined on $\overline{\mathrm{B}_{1/5}}$.

The obtained estimations on $h_\alpha$ are going to be used in the sequel. These were
\begin{equation*}
|h_{20}(\alpha)|/2+|h_{11}(\alpha)|+|h_{02}(\alpha)|/2<0.76,
\end{equation*}
\begin{equation*}
|h_{30}(\alpha)|/6 + |h_{12}(\alpha)|/2 + |h_{03}(\alpha)|/6<0.52
\end{equation*}
and in particular
\begin{equation}\label{h_up_low}
|w|-0.76|w|^2-0.52|w|^3<|h_{\alpha}(w)|<|w|+0.76|w|^2+0.52|w|^3
\end{equation}
for all $\alpha \in [0.999,1]$.

Let $H=\{(\alpha,w)\in\R\times\C:\ \alpha\in[0.999,1],\ w\neq 0 \text{ and } |w|<1/20\}$. In the sequel, we shall always assume that $(\alpha,w)\in H$. From this assumption and inequality \eqref{h_up_low} we readily get that  $|w|<1.05|h_\alpha(w)|$ from which we get in particular that $|z|=|h_{\alpha}(w)|<1/19$.

Our goal now is to give $\eps_0\in (0,1/20],$ independent of $\alpha$ such that for every $\alpha \in [0.999,1]$, if $0<|w|<\eps_0$, then $\left|h^{-1}_\alpha\left(G\left(h_\alpha(w),\overline{h_\alpha(w)},\alpha\right)\right)\right|<|w|$ holds which guarantees that $\mathrm{B}_{\eps_0}$ belongs to the basin of attraction of the fixed point  $0$ of the discrete dynamical system generated by \eqref{Ricker_complex_w}. For, we turn our attention to the estimation of function $R_2$ in \eqref{Ricker_complex_w}. 

First, we go back to \eqref{Ricker_complex_z}. Let us consider
$$g(z,\overline{z},\alpha)=\sum_{k+l=2,3}\frac{g_{kl}(\alpha)}{k!l!}z^k\overline{z}^l+ R_1(z,\overline{z},\alpha).$$
The explicit formulae for $g_{20}(\alpha),\dots,g_{03}(\alpha)$ can be found in equations \eqref{app_g20}--\eqref{app_g03}. By interval arithmetics, one may obtain that for all $\alpha\in [0.999,1]$,
\begin{equation}\label{g2nd}
\frac{|g_{20}(\alpha)|}{2}+|g_{11}(\alpha)|+\frac{|g_{02}(\alpha)|}{2}=\\
\frac{1-\alpha +\sqrt{\alpha  (2+\alpha )}}{\sqrt{\alpha  (-1+4 \alpha )}}<1.01,
\end{equation}

\begin{equation}\label{g3rd}
\frac{|g_{30}(\alpha)|}{6}+\frac{|g_{21}(\alpha)|}{2}+\frac{|g_{12}(\alpha)|}{2}
+\frac{|g_{03}(\alpha)|}{6}=
\sqrt{\frac{(6+\alpha )}{9\alpha(4 \alpha-1) }}+\sqrt{\frac{2+(\alpha-2) \alpha }{\alpha^2(4 \alpha-1) }}<1.09.
\end{equation}

We also have that $R_1(z,\overline{z},\alpha)=\sum_{k+l=
4}\frac{g_{kl}(\alpha)}{k!l!}z^k\overline{z}^l+\tilde{R_1}(z,\overline{z},\alpha)$, where $\tilde{R_1}(z,\overline{z},\alpha)=O(|z|^5)$ for fixed $\alpha$. For explicit formulae of the fourth order coefficients see equations \eqref{app_g40}-\eqref{app_g04} in the appendix. It is clear from equations \eqref{Ricker_beta_F}, \eqref{uv(z)} and \eqref{Ricker_complex_z} that 
$$\tilde{R_1}(z,\overline{z},\alpha)=\frac{\sqrt{4\alpha-1}-i}{2\sqrt{4\alpha-1}}
\left(v\sum_{k=4}^{\infty}\frac{(-u)^k}{k!}+
\alpha\sum_{k=5}^{\infty}\frac{(-u)^k}{k!}\right),$$
where $u$ and $v$ are defined by equation \eqref{uv(z)}. Using $0<|z|<1/19$ and \eqref{u,v<z} we have that $|u|<1/8,\ |v|<1/8$ and obtain
\begin{equation*}
\begin{split}
\left|\tilde{R_1}(z,\overline{z},\alpha)\right|&\leq
\sqrt{\frac{\alpha}{4\alpha-1}}\left(|v|\frac{e^{1/8}}{4!}|u|^4+
\alpha\frac{e^{1/8}}{5!}|u|^5\right)<\sqrt{\frac{\alpha}{4\alpha-1}}\frac{8}{7}\left(2|z|\frac{16}{24\alpha^2}+
\frac{32}{\alpha^{3/2}120}|z|\right)|z|^4\\
&\leq \sqrt{\frac{\alpha}{4\alpha-1}}\frac{8}{7}\left(\frac{4}{57\alpha^2}+
\frac{4}{285\alpha^2}\right)|z|^4=\frac{64}{665}\sqrt{\frac{1}{(4\alpha-1)\alpha^3}}|z|^4.
\end{split}
\end{equation*}
Now for all $(\alpha,w)\in H$ with $z=h_\alpha(w)$ we get that
\begin{equation*}
\begin{split}
\left|R_1(z,\overline{z},\alpha)\right| &\leq \sum_{k+l=
4}\left|\frac{g_{kl}(\alpha)}{k!l!}z^k\overline{z}^l\right|+\left|\tilde{R_1}(z,\overline{z},\alpha)\right|=\frac{6-3 \alpha +\sqrt{\alpha  (12+\alpha )}+4 \sqrt{3+\alpha ^2}}{12 \sqrt{\alpha ^3 (-1+4 \alpha )}}+\left|\tilde{R_1}(z,\overline{z},\alpha)\right|\\
&<\frac{4758-1995 \alpha +665 \sqrt{\alpha  (12+\alpha )}+2660 \sqrt{3+\alpha ^2}}{7980 \sqrt{\alpha ^3 (-1+4 \alpha )}}.
\end{split}
\end{equation*}
By interval arithmetics we obtain for all $\alpha\in[0.999,1]$ that
\begin{equation}\label{g4th}
\left|R_1(z,\overline{z},\alpha)\right|<0.76|z|^4.
\end{equation}

We also need a similar estimation on $\left|h_{\alpha}^{-1}(z)\right|$. Let us recall that $z=h_{\alpha}(w)$ and 
$$w=h_{\alpha}^{-1}(z)=h_{inv,\alpha}(z)+R_3(z,\alpha).$$
As $h_{inv,\alpha}(z)$ is a polynomial of $z$ and $\overline{z}$ of degree four (see formulae \eqref{app_hinv(z)Hij}--\eqref{app_h03}), we denote the coefficient corresponding to $z^k\overline{z}^l$ by $h_{inv}^{kl}(\alpha)$. Calculating these coefficients and using interval arithmetics, we obtain that for all $\alpha \in [0.999,1],$
\begin{equation}\label{hinv2-4}
\sum_{k+l=2}\left|h_{inv}^{kl}\right|<0.76,\quad \sum_{k+l=3}\left|h_{inv}^{kl}\right|<1.06 \quad \text{and} \quad \sum_{k+l=4}\left|h_{inv}^{kl}\right|<1.39.
\end{equation}
See formulae \eqref{app_hinv2} -- \eqref{app_hinv_04}. Let us recall that for all $(\alpha,w)\in H$, $|w|<1.05|h_\alpha(w)|$ holds. Now, for the fifth and higher order terms in $R_3$, first we give an estimation of type $|R_3(h_\alpha(w),\alpha)|<K_3 |w|^4$ and then we get that $|R_3(z,\alpha)|<K_3 1.05^4 |z|^4$, with $z=h_\alpha(w)$.

From the definition of $h_{inv,\alpha}$, it follows that $R_3(h_{\alpha}(w),\alpha)=w-h_{inv,\alpha}(h_{\alpha}(w))$
is a polynomial of $w$ and $\overline{w}$ and it has only fifth and higher order terms. Let $r_3^{kl}(\alpha)$ denote the coefficient of  $R_3(h_{\alpha}(w),\alpha)$ corresponding to $w^k\overline{w}^l$. We use a bit rougher estimation for $\left|\sum_{5\leq k+l}r_3^{k+l}(\alpha)w^k\overline{w}^l\right|$. Namely, first we give the  estimations 
\begin{equation}\label{hnm_up}
\begin{split}
&|h_{20}(\alpha)|<1.01;  \quad |h_{11}(\alpha)|<0.001;  \quad |h_{02}(\alpha)|<0.51\\
&|h_{30}(\alpha)|<0.89; \quad |h_{12}(\alpha)|<0.45; \quad |h_{03}(\alpha)|<0.89
\end{split}
\end{equation}
for all $\alpha\in[0.999,1]$. Now, in $R_3(h_{\alpha}(w),\alpha)$ we replace $w$ and $\overline{w}$ by $|w|$, $h_{nm}(\alpha)$ by the estimates given in \eqref{hnm_up} (for $2\leq n+m\leq 3,\ (n,m)\neq(2,1))$, and then we turn every $-$ sign into $+$ to get a real polynomial $\hat{R}_3(|w|)$, with nonnegative coefficients $\hat{r}_3^k$ (independent of $\alpha$) corresponding to $|w|^k$. If we use $0<|w|<1/20$, then we get that
\begin{equation}\label{R3_5}
\sum_{5\leq k+l\leq 12}|r_3^{kl}(\alpha)||w|^{k+l}<\sum_{5\leq k\leq 12}\hat{r}_3^k|w|^k<\sum_{5\leq k\leq 12}\hat{r}_3^k|w|^4\left(\frac{1}{20}\right)^{k-4}<1.02|w|^4.
\end{equation}
This implies that
\begin{equation}\label{R3_4...}
|R_3(z,\alpha)|<1.05^4\cdot 1.02|z|^4<1.24|z|^4
\end{equation}
for all $(\alpha,w)\in H$, where $z=h_\alpha(w)$. It is now clear that
\begin{equation}\label{hinvup}
h_{inv}(z)<|z|+0.76|z|^2+1.06|z|^3+2.63|z|^4
\end{equation}
holds for all $(\alpha,w)\in H$ with $z=h_\alpha(w)$.

Now, we are able to give an estimation on $R_2$ in \eqref{Ricker_complex_w}. First, according to our previous estimations, let us define the following real polynomials.
$$h^{max}(s)=s+0.76s^2+0.52s^3,$$
$$G^{max}(s)=s+1.01s^2+1.09s^3+0.76s^4,$$
$$h_{inv}^{max}(s)=s+0.76s^2+1.06s^3+2.63s^4$$
and
$$Q(s)=\sum_{k=1}^{48}q_{k}(s)=h_{inv}^{max}\circ G^{max} \circ h^{max}(s).$$
By now, it is obvious that for $(\alpha,w)\in H$, $\left|R_2(w,\overline{w},\alpha)\right|<\sum_{k=4}^{48}q_k(|w|)|w|^4\left(\frac{1}{20}\right)^{k-4}$ holds, which leads to $\left|R_2(w,\overline{w},\alpha)\right|<23.9|w|^4$. For our purposes this approach is too rough -- the obtained neighbourhood would be too small $(\approx \mathrm{K}(\underline{\alpha}; 1/80))$ and we could not prove that every trajectory enters it eventually. Hence we have to be as sharp as we can in our estimations to obtain as large neighbourhood as possible. So, instead of only estimating these functions separately, let us consider the composite function $h_{\alpha}^{-1}\circ G_\alpha \circ h_\alpha$, where $G_\alpha(z)$ denotes $G(z,\overline{z},\alpha)$. Now, we are only interested in the fourth and higher order terms. Since $h_\alpha$ is a known function and we also know functions $G_{\alpha}$ and  $h_{\alpha}^{-1}$ up to fourth order terms, hence we are able to compute the fourth order coefficients of $h_{\alpha}^{-1}\circ G_\alpha \circ h_\alpha$, denoted by $r_2^{kl}(\alpha)$, where $k+l=4$. By interval arithmetics we show that
\begin{equation}\label{R2pt1}
\sum_{k+l=4}|r_2^{kl}(\alpha)|<1.02.
\end{equation}
See equations \eqref{app_r2_40} -- \eqref{app_r2_04} for the formulae. Using this we infer that 
\begin{equation}\label{R2}
|R_2(w,\overline{w},\alpha)|<1.02|w|^4+
\sum_{k=5}^{48}q_k|w|^4\left(\frac{1}{20}\right)^{k-4}<4.6|w|^4.
\end{equation}
for all $(\alpha,w)\in H$.

Now, we turn our attention to $c_1(\alpha)$ in \eqref{Ricker_complex_w}. The formula for $c_1(\alpha)$ can be found in \eqref{app_c1}.

According to \cite{Kuznetsov} and using inequality \eqref{R2} we get the following
\begin{equation}\label{radialest}
\begin{split}
&\left|\mu(\alpha)w+c_1(\alpha)w^2\overline{w}+R_2(w,\overline{w},\alpha)\right|\leq
|w|\left|\mu(\alpha)+c_1(\alpha)|w|^2\right|+\left|R_2(w,\overline{w},\alpha)\right|\\
&=|w|\left| |\mu(\alpha)|+d(\alpha)|w|^2\right|+\left|R_2(w,\overline{w},\alpha)\right|<
|w|\left| \sqrt{\alpha}+d(\alpha)|w|^2\right|+4.6|w|^4,
\end{split}
\end{equation}
for all $(\alpha,w)\in H$, where $d(\alpha)=\frac{|\mu(\alpha)|}{\mu(\alpha)}c_1(\alpha)$. Let 
$$R_4(w,\alpha)=\left|\sqrt{\alpha}+d(\alpha)|w|^2\right|-(\sqrt{\alpha}+a(\alpha)|w|^2),$$
where $a(\alpha)$ denotes the real part of $d(\alpha)$.

In the following we are going to prove that $\left|R_4(w,\alpha)\right|< 0.1|w|^3$ holds for all $(\alpha,w)\in H$. First of all, the formula for function $a$ is the following
\begin{equation}
a(\alpha)=\frac{4+\alpha  \left(-10+\alpha +\alpha ^2\right)}
{4 \alpha ^{3/2} (-1+\alpha(4+\alpha ))}.
\end{equation}
It can be readily shown that
\begin{equation}\label{a_up_low}
-1< a(\alpha)\leq -\frac{1}{4}
\end{equation}
holds for all $\alpha\in[0.999,1]$. Using the definition of $d$ and $a$, the estimation above and the assumption $(\alpha,w)\in H$ we get the following.
\begin{equation*}
\begin{split}
&\left|\left|\sqrt{\alpha}+d(\alpha)|w|^2\right|-(\sqrt{\alpha}+a(\alpha)|w|^2)\right|=
\left|\sqrt{\alpha+2\sqrt{\alpha}a(\alpha)|w|^2+|d(\alpha)|^2|w|^4}-(\sqrt{\alpha}+a(\alpha)|w|^2)\right|\\
&=\left|\frac{(|d(\alpha)|^2-(a(\alpha))^2)|w|^4}
{\sqrt{\alpha+2\sqrt{\alpha}a(\alpha)|w|^2+|d(\alpha)|^2|w|^4}+\sqrt{\alpha}+a(\alpha)|w|^2}\right|\\
&\leq\frac{(|d(\alpha)|^2-(a(\alpha))^2)|w|^4}
{\sqrt{400\alpha|w|^2+2\sqrt{\alpha}a(\alpha)|w|^2}+20\sqrt{\alpha}|w|+a(\alpha)|w|}\leq
\frac{(|d(\alpha)|^2-(a(\alpha))^2)}{\sqrt{400\alpha-2}+20\alpha-1}|w|^3\\
&\leq\frac{(|d(\alpha)|^2-(a(\alpha))^2)}{19\sqrt{\alpha}+18\alpha}\cdot|w|^3\leq
\frac{(|d(\alpha)|^2-(a(\alpha))^2)}{37\alpha}\cdot|w|^3=
\frac{\left(\alpha^4+3\alpha^3-12\alpha^2+20\alpha-4\right)^2}{16\cdot 37\alpha^6(4\alpha-1)(\alpha^2+4\alpha-1)^2}\cdot|w|^3.
\end{split}
\end{equation*}
From this last formula it can be proven that $|R_4(w,\alpha)|<0.1|w|^3$. From inequalities  \eqref{radialest}, \eqref{a_up_low} and from the above estimate we obtain that for all $(\alpha,w)\in H$ the inequalities
\begin{equation}\label{righths<w}
\begin{split}
&\left|\mu(\alpha)w+c_1(\alpha)w^2\overline{w}+R_2(w,\overline{w},\alpha)\right|<
|w|\left(\sqrt{\alpha}+a(\alpha)|w|^2+R_4(w,\alpha)\right)+4.6|w|^4\\
&<\sqrt{\alpha}|w|-\frac{1}{4}|w|^3+4.7|w|^4<|w|\left(1-\frac{1}{4}|w|^2(1-4\cdot 4.7|w|)\right)<|w|
\end{split}
\end{equation}
hold provided that $|w|<\frac{1}{20}=\min\left\{\frac{1}{20},\frac{1}{4\cdot 4.7}\right\}$. This means that for all $\alpha \in [0.999,1]$ the set  $\mathrm{B}_{1/20}$ belongs to the basin of attraction of the discrete dynamical system generated by \eqref{Ricker_complex_w}.

Now, we only have to show that for any $(x,y)^T \in K(\underline{\alpha}; 1/22)$, after our transformations, $|w|<\eps_0=\frac{1}{20}$ holds. First, we need $\eps_G$ such that for $|z|<\eps_G$, $|w|=|h_\alpha^{-1}(z)|<\frac{1}{20}$ holds for all $\alpha \in [0.999,1]$. From \eqref{h_up_low} we obtain that $\eps_G=\frac{1}{21}$ is an appropriate choice. Now, if $|u|<\eps,\ |v|<\eps$, then from \eqref{z<deltaalpha} we get that
\begin{equation*}
|z(u,v,\alpha)|\leq \eps\sqrt{\frac{\alpha(2+\alpha)}{4\alpha-1}} .
\end{equation*}
Thus for
\begin{equation}\label{eps}
\eps<\min_{\alpha\in[0.999,1]}\sqrt{\frac{4\alpha-1}{\alpha(2+\alpha)}} \cdot\frac{1}{21}
\end{equation}
we get that if $|u|<\eps$ and $|v|<\eps$ then $|z(u,v,\eps)|<\frac{1}{21}$. It is easily shown that $\eps=\frac{1}{22}$ fulfils inequality \eqref{eps} which proves our proposition.
\end{proof}

\begin{proposition}\label{stable}
Fixed point $\underline{\alpha}$ of map $F_\alpha$ is locally asymptotically stable if and only if $\alpha \in (0,1]$.
\end{proposition}
\begin{proof}
By linearisation one readily gets that $\underline{\alpha}$ is locally asymptotically stable if $\alpha\in(0,1)$, and unstable if $\alpha\in(1,\infty)$. Transforming map $F_\alpha$ to the form \eqref{Ricker_complex_w} and using inequality \eqref{righths<w} yields that $\underline{1}$ is a locally asymptotically stable fixed point of $F_1$.
\end{proof}

\section{Graph representations}
\label{graph}

\subsection*{Covers and graph representations}
Different directed graphs can be associated with a given map. The graphs reflect the behaviour of the map up to a given resolution. The vertices of these graphs are sets and the edges correspond to transitions between them. We can derive properties of our dynamical system through the study of the graphs. These techniques appeared in many articles, in both rigorous and non-rigorous computations for maps by Hohmann, Dellnitz, Junge, Rumpf \cite{Hohmann_Subdivision}, \cite{Hohmann_Invariant}, Galias \cite{Galias_Ikeda}, Luzzatto and Pilarczyk \cite{Luzzatto_Finite}, and computations for the time evolution of a continuous system with a given timestep by Wilczak  \cite{Wilczak_Kuznetsov}. We introduce the general setting and two applications in particular. One to enclose the non-wandering points and the other one to estimate the basin of attraction.

Both methods (Algorithms \ref{graph-alg-nonw} and \ref{graph-alg-basin}) combined with local estimations of the type of Section \ref{neighbourhood} at the critical points, can be theoretically applied to prove different dynamical properties. On the one hand these algorithms are included for possible future references, on the other hand certain elements of these algorithms proved to be useful for the map $F_\alpha$ in Section \ref{application}. In particular, the correctness of Algorithm \ref{graph-alg-nonw} is crucial in Section \ref{application}.

\begin{definition}
  \label{graph-def-cover}
  $\S$ is called a \emph{cover} of $\domain \subseteq \R^n$ if it is a collection of subsets of $\R^n$ such that $\cup_{s \in \S}s \supseteq \domain$. We denote their union $\cup_{s \in \S} s$ by $|\S|$ in the following.
  We define the \emph{diameter} or \emph{outer resolution} of the cover $\S$ by
  \begin{equation*}
    \resR^+(\S) = \diam(\S) = \sup_{s \in \S} \diam(s).
  \end{equation*}
  with
  \begin{equation*}
  \diam(s) = \sup_{x, y \in s} \norm{x - y}. 
  \end{equation*}
  A cover $\S_2$ is said to be \emph{finer} than the cover $\S_1$ if
  \begin{equation*} 
    \left( \forall s_1 \in \S_1 \right) \; \left( \exists \{ s_{2,i}, i \in \I \} \subseteq \S_2 \right) \; \textrm{ such that } \bigcup_{i \in \I} s_{2,i} =s_1.
  \end{equation*}
  We denote this relation by $\S_2 \preccurlyeq \S_1$.
  The \emph{inner resolution} of a cover $\S$ is the following:
  \begin{equation*}
    \resR^{-}(\S) = \sup \{ r \geq 0 : \forall x \in \domain, \exists s \in \S : \Ball(x; r) \subseteq s\}.
  \end{equation*}
  A cover $\S$ is \emph{essential} if $\S \setminus s$ is not a cover anymore for any $s \in \S$.
  The cover $\P$ is called a \emph{partition} if it consists of closed sets such that $|\P| = \domain$ and $\forall p_1, p_2 \in \P : p_1 \cap p_2 \subseteq \bd(p_1) \cup \bd(p_2)$, where $\bd(p)$ is the boundary of the set $p$. Consequently, for any partition $\P$ the inner resolution $\resR^{-}(\P)$ is zero.
\end{definition}

In the following we will always work with essential and finite partitions, as a consequence the supremum in the definition of the diameter $\resR^+$ of the partition becomes a maximum.

\begin{definition}
   A \emph{directed graph} $\G = \G(\V,\E)$ is a pair of sets representing the vertices $\V$ and the edges $\E$, that is: $\E \subseteq \V \times \V$, and $(u,v) \in \E$ means that $\G$ has a directed edge going from $u$ to $v$. We say that $v_1 \to v_2 \to \ldots \to v_k$ is a \emph{directed path} if $(v_i, v_{i+1}) \in \E$ for all $i = 1, \ldots, k - 1$. If $v_k = v_1$, then it is a \emph{directed cycle}.

  A directed graph $\G$ is called \emph{strongly connected} if for any $u, v \in \V$, $v \neq u$ there is a directed path from $u$ to $v$ and from $v$ to $u$ as well. The \emph{strongly connected components} (SCC) of a directed graph $\G$ are its maximal strongly connected subgraphs. It is easy to see that $u$ and $v$ are in the same SCC if and only if there is a directed cycle going through both $u$ and $v$. Every directed graph $\G$, can be decomposed into the union of strongly connected components and directed paths between them. If we contract each SCC to a new vertex, we obtain a directed acyclic graph, that is called the \emph{condensation} of $\G$.

  We say that the directed paths $p_1, p_2$ are from the same \emph{family of directed paths}, if they visit the same vertices in $\V$ (multiple visits are possible). If the set of the visited vertices is $V \subseteq \V$, then we denote the family by $\Upsilon_{path}(V)$, and we say that $V$ is the \emph{vertex set} of the family. In a similar manner we can define the \emph{family of directed cycles}, and denote it by $\Upsilon_{cycle}(V)$, and say that $V$ is the vertex set of the family.
\end{definition}

\begin{definition}
  \label{graph-def-repr}
  Let $f : \domain_f \subseteq \R^n \to \R^n$, $\domain \subseteq \domain_f$, and $\S$ be a cover of $\domain$.
  We say that the directed graph $\G(\V,\E)$ is a \emph{graph representation of $f$ on $\domain$ with respect to $\S$}, if there is a $\iota : \V \to \S$ bijection such that the following implication is true for all $u, v \in \V$:
  \begin{equation*}
    f( \iota(u) \cap \domain ) \cap \iota(v) \cap \domain \neq \emptyset \Rightarrow (u,v) \in \E,
  \end{equation*}
  and we denote it by $\G \propto (f, \domain, \S)$.
\end{definition}

Having a graph representation $\G$ of $f$ on $\domain$ with respect to $\S$, we take the liberty to handle the elements of the cover as vertices and vica versa, omitting the usage of $\iota$.
It is important to emphasise that in general $(u,v) \in \E$ does not imply that $f(u \cap \domain) \cap v \cap \domain \neq \emptyset$. If we have $(u,v) \in \E \Leftrightarrow f(u \cap \domain) \cap v \cap \domain \neq \emptyset$, then we call $\G$ an \emph{exact} graph representation.

\subsection*{Enclosure of the non-wandering points}
Consider the continuous map
\begin{equation}
  \label{graph-eq-discrete}
  f : \domain_f \subseteq \R^n \to \R^n
\end{equation}
Let $f^{-1}(x) = \{ y \in \domain_f : f(y) = x \}$, for $x \in \R^n$.
\begin{definition}
  \label{graph-def-nonw}
  The point $q \in \domain_f$ is called a \emph{fixed point} if $f(q) = q$. $q \in \domain_f$ is a \emph{periodic point with minimal period $m$} if $f^m(q) = q$ and for all $0 < k < m : f^k(q) \neq q$; $q \in \domain_f$ is \emph{eventually periodic} if it is not periodic, but there is a $k_0$ such that $f^{k_0}(q)$ is periodic. 
  The point $q \in \domain_f$ is a \emph{non-wandering point} of $f$ if for every neighbourhood $U$ of $q$ and for any $M \geq 0$, there exists an integer $m \geq M$ such that $f^m(U \cap \domain_f) \cap U \cap \domain_f \neq \emptyset$.

  Let $K \subseteq \domain_f$ be a compact set. We denote the set of periodic points of $f$ in $K$ by $\Per{f}{K}$ and the set of non-wandering points of $f$ in $K$ by $\NonW{f}{K}$.
\end{definition}
Instead of directly studying the map (\ref{graph-eq-discrete}), we will analyse different graph representations of $f$.
Let $K \subseteq \domain_f$ be a compact set satisfying $f(K) \subseteq K$, $\P$ a partition of $K$, and $\G \propto (f, K, \P)$. We shall use the algorithm from \cite{Galias_Ikeda} to enclose the non-wandering points in $K$.
\begin{algorithm}[H]
\caption{Enclosure of non-wandering points} \label{graph-alg-nonw}
\begin{algorithmic}[1]
\Function{EnclosureNonW}{$f, K, \delta_0$}
  \Comment $f$ is the function, $K$ is the starting region.
  \State $k \leftarrow 0$
  \State $\V_0 \leftarrow$ Partition($K, \delta_0$)
  \Comment $\V_0$ is a partition of $K$, $\diam(\V_0) \leq \delta_0$.
  \Loop  
    \State $\E_k \leftarrow$ Transitions($\V_k, f$)
    \Comment The possible transitions (extra edges may occur).
    \State $\G_k \leftarrow$ GRAPH$(\V_k, \E_k$)
    \Comment $\G_k \propto (f, |\V_k|, \V_k)$
    \ForAll {$v \in \V_k$}
      \If {$v$ is \textbf{not} in a directed cycle}\label{graph-line-nonw-check}
	  \State \textbf{remove} $v$ from $\G_k$
      \EndIf
    \EndFor
    \If {$\G_k$ is empty}
      \State \Return $\emptyset$
      \Comment There is no non-wandering point in $K \Rightarrow \NonW{f}{K} = \emptyset$.
    \Else
      \State $\delta_{k+1} \leftarrow$ $\delta_k / 2$
      \State $\V_{k+1} \leftarrow$ Partition($|\V_k|, \delta_{k+1}$)
      \Comment $\V_{k+1}$ is a partition of $|\V_k|$, $\diam(\V_{k+1}) \leq \delta_{k+1}$
      \State $k \leftarrow k + 1$
    \EndIf
  \EndLoop
\EndFunction
\end{algorithmic}
\end{algorithm}
However, this algorithm for enclosing non-wandering points appeared without a full proof. We will give it here, not just for the sake of completeness, but because some steps are non-trivial. We need to take special care when a non-wandering point is on the boundary of a partition element.

For any $x \in K$, define
\begin{equation*}
  \tilde{\P}_x := \{ u \in \P : x \in u \}.
\end{equation*}
Since we are working with finite covers, both $\P$ and $\tilde{\P}_x$ are finite.

\begin{lemma}
  \label{graph-lemma-nonw1}
  For every $q \in K$, there is an $\eta_q > 0$ such that for any $u \in \P$, if $u \cap \Ball(q; \eta_q) \neq \emptyset$, then $q \in u$.
\end{lemma}

\begin{proof}
  Since we are working with finite partitions, this is very easy to see. If $\tilde{\P}_q = \P$ then any positive number satisfies the condition. Otherwise define
  \begin{equation*}
    \eta := \min_{u \in \P \setminus \tilde{\P}_q} d(q, u).
  \end{equation*}
  This is positive, since $\P \setminus \tilde{\P}_q$ is a finite set and for $u \in \P \setminus \tilde{\P}_q$, $u$ is compact and $q \notin u$. Now we can take any number for $\eta_q$ from $(0, \eta)$.
\end{proof}

\begin{lemma}
  \label{graph-lemma-nonw2}
  For every $q \in \NonW{f}{K}$, there are $u, v \in \P$ such that
  \begin{enumerate}
   \item $q \in u \cap v$
   \item for every $\varepsilon > 0$, there are $x = x(\varepsilon), N = N(x(\varepsilon), \varepsilon) \in \N$ such that $\lim_{\varepsilon \to 0}N(x(\varepsilon), \varepsilon) = \infty$, $x \in \Ball(q; \varepsilon), f^N(x) \in \Ball(q; \varepsilon)$, $x \in u$ and $f^N(x) \in v$.
  \end{enumerate}
\end{lemma}

\begin{proof}
  Consider a decreasing sequence of positive numbers $\{ \varepsilon_k \}^\infty_{k = 0}$ with $\lim_{k \to \infty}\varepsilon_k = 0$ and $\varepsilon_0 < \eta_q$. Since $q$ is non-wandering,
  \begin{equation*}
    \exists N_k \geq k : f^{N_k}( \Ball(q; \varepsilon_k) \cap \domain_f ) \cap \Ball(q; \varepsilon_k) \cap \domain_f \neq \emptyset.
  \end{equation*}
  Therefore
  \begin{equation*}
    \exists x_k \in \Ball(q; \varepsilon_k) : f^{N_k}( x_k ) \in \Ball(q; \varepsilon_k).
  \end{equation*}
  Since $\tilde{\P}_q$ is finite, we can pick $u \in \tilde{\P}_q$ such that it contains infinitely many $x_k$-s. We may assume, by switching to an appropriate subsequence and reindexing, that $\forall k : x_k \in u$. It is now possible to choose -- because of finiteness again -- an element of the partition $v \in \tilde{\P}_q$ such that it contains infinitely many of $f^{ N_k }( x_k )$. Switching again to the subsequence, the required conditions are now satisfied.
\end{proof}

\begin{remark}
  If there exists $u \in \P$ such that $q \in \int(u)$ then it follows from the definition, that $u$ and $v = u$ is a good choice. By $\int(u)$ we mean the interior of the set $u$.
\end{remark}

\begin{lemma}
  \label{graph-lemma-nonw3}
  For every $q \in \NonW{f}{K} \setminus \Per{f}{K}$ there is an element $u \in \tilde{\P}_q$ and a family of directed cycles $\Upsilon_{cycle}(V)$ in $\G$ such that $u \in V$, and the family encloses infinitely many trajectories in $K$ of the form $\left( x_k, f(x_k), f^2(x_k), \ldots, f^{N_k} (x_k) \right)$ with
  \begin{equation*}
    \lim_{k \to \infty} N_k = \infty \; \textrm{ and } \lim_{k \to \infty} x_k = \lim_{k \to \infty} f^{N_k} (x_k) = q.
  \end{equation*}
\end{lemma}

\begin{proof}
  If there is a $u \in \P$ such that $q \in \int(u)$, then since $q \in \NonW{f}{K}$, there are infinitely many such trajectories, and each one of them is enclosed by a directed cycle that passes through $u$. Since there is only finite number of families of directed cycles,  therefore we can pick one family $\Upsilon_{cycle}(V)$ that encloses infinitely many trajectories, and $u \in V$.

  If such $u$ cannot be found, then we will do the following: for a series of positive numbers $\lim_{k \to \infty} \varepsilon_k = 0$, with the use of Lemma \ref{graph-lemma-nonw2}, we obtain the sets $u, v \in \tilde{P}_q$, the points $x_k \in u, f^{N_k}(x_k) \in v$ with $N_k \to \infty$, and that $\lim_{k \to \infty} x_k = \lim_{k \to \infty} f^{N_k} (x_k) = q$.

  Since $\{ f^{N_k - 1}(x_k) \} \subseteq K$, and $K$ is compact, we may assume that this sequence converges to a point $q' \in K$, that is $f^{N_k - 1}(x_k) \to q'$. From the continuity of $f$ it follows that
  \begin{equation*}
    f(q') = \lim_{k \to \infty} f^{N_k} (x_k) = q.
  \end{equation*}
  Since $\tilde{\P}_{q'}$ is finite, and because of Lemma \ref{graph-lemma-nonw1}, infinitely many of the points $f^{N_k - 1}(x_k)$ are inside one of its elements, without loss of generality we may assume, that
  \begin{equation*}
    \forall k \in \N : f^{N_k - 1}(x_k) \in v' \in \tilde{\P}_{q'}.
  \end{equation*}
  With similar argument we may assume that
  \begin{equation*}
    \forall k \in \N : f(x_k) \in u' \in \tilde{\P}_{f(q)}.
  \end{equation*}
  We are thus considering directed paths in the graph of the following form:
  \begin{equation*}
    u \to u' \to \ldots \to v' \to v.
  \end{equation*}
  Since there is only finite number of families of directed paths in $\G$ and infinitely many trajectories of the desired property are enclosed by them, there must be at least one such family $\Upsilon_{path}(V)$ that encloses infinitely many trajectories itself. The directed paths $v' \to v \to u'$ and $v' \to u \to u'$ are present in $\G$, since $f(q') = q$. The situation is depicted in Figure \ref{graph-fig-nonw}.
  \begin{figure}[!ht]
    \centering
    \includegraphics[scale=0.8]{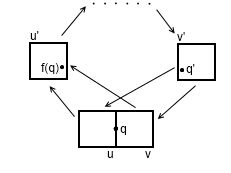}
    \caption{\label{graph-fig-nonw}The directed paths close to $q$}
  \end{figure}

  We can conclude our argument that the set $V$ is a vertex set for a family of directed cycles as well, since
  \begin{equation*}
    u \to u' \to \ldots^{\textrm{vertices from } V} \ldots \to v' \to v \to u' \to \ldots^{\textrm{vertices from } V} \ldots \to v' \to u,
  \end{equation*}
  is a directed cycle. Thus $\Upsilon_{cycle}(V)$ is well defined and encloses infinitely many trajectories of the desired type.
\end{proof}
Now we are ready to prove the correctness of Algorithm \ref{graph-alg-nonw}.

\begin{theorem}
  \label{graph-thm-nonw}
  If $\NonW{f}{K} \neq \emptyset$, then Algorithm \ref{graph-alg-nonw} will never stop and $\NonW{f}{K} \subseteq \V_k$ is satisfied for every $k$.
\end{theorem}

\begin{proof}
  Assume that $q \in \NonW{f}{K}$. If $q$ is a periodic point with period $N + 1$, then its orbit\\
  $\left\{q, f(q), \ldots, f^N(q)\right\}$ is enclosed in a directed cycle and none of these vertices is removed in the first step. Therefore $q$ is a periodic point of $f$ restricted to $|\V_1|$ as well. Repeating the same argument gives that $q$ is always enclosed.

  If $q \in \NonW{f}{K}$, but it is not periodic, then we obtain a family of directed cycles $\Upsilon_{cycle}(V)$ from Lemma \ref{graph-lemma-nonw3}, that encloses infinitely many trajectories of the type mentioned before. The vertices in $V$ are not removed, since they are in a directed cycle, thus the enclosures for all of these trajectories are preserved. Therefore $q$ is a non-wandering point of $f$ restricted to $|\V_1|$ as well. Since $|\V_k|$ is compact, we can repeat the same argument and obtain families of directed cycles $\Upsilon_{cycle}(V_k)$, that enclose infinitely many of these trajectories, that ensure that $q$ stays a non-wandering point when we restrict $f$ to $|\V_k|$ and $q \in |\V_k|$ is satisfied.
\end{proof}

\begin{remark}
  The theorem does not imply that at one step each vertex containing $q$ is kept.
\end{remark}

\begin{remark}
  We decide whether a vertex is in a directed cycle by decomposing the graph into strongly connected components. The vertices that form a component by themselves and have no self edges are the ones that are not in directed cycles. To find this decomposition, we will use the algorithm of Tarjan \cite{Tarjan_SCC}, that runs in linear time.
\end{remark}

\subsection*{Inner enclosure of the basin of attraction}
Consider the continuous map
\begin{equation}
  \label{graph-eq-discrete-2}
  f : \domain_f \subseteq \R^n \to \R^n.
\end{equation}
\begin{definition}
 A set $\O \subseteq \domain_f$ is called \emph{invariant} if $f(\O) = \O$. An invariant set $\O$ is called an \emph{attracting set} if there exists an open neighbourhood $\U \subseteq \domain_f$ of $\O$ such that
  \begin{equation}
    \label{graph-eq-trapping}
    \left( \forall \textrm{ open neighbourhood } V \supseteq \O \right) \; \left( \exists \, L(V) \geq 0 \right) \; \textrm{ such that } \forall k \geq L(V) : f^k(\U) \subseteq V
  \end{equation}
  This neighbourhood $\U$ is called a \emph{fundamental neighbourhood} of $\O$. The \emph{basin of attraction} of $\O$ is $\cup_{k \in \N_0} f^{-k}(\U)$. 
\end{definition}
Assume now that $K \subseteq \domain_f$ is a compact set such that $f(K) \subseteq K$. When analysing the forward orbits starting from $K$ we can work with $f|_K$. Assume that there is an attracting set $\O \subseteq K$ for $f|_K$ and a neighbourhood $U$ such that $\O \subseteq U \subseteq K$ and $U$ is inside the basin of attraction of $\O$. We want to find a set $B$ in the basin of attraction of $\O$ so that $U \subset B$. We will use the algorithm from \cite{Galias_Ikeda}:

\begin{algorithm}[H]
\caption{Inner enclosure of the basin of attraction} \label{graph-alg-basin}
\begin{algorithmic}[1]
\Function{Basin\_of\_Attraction}{$f, K, \delta_0; U$}
  \Comment $U$ is the initial attracting neighbourhood.
  \State $k \leftarrow 0$
  \State $W \leftarrow \emptyset$
  \Comment We collect the vertices in the basin of attraction into $W$.
  \State $\V_0 \leftarrow$ Partition($K, \delta_0$)
  \Comment $\V_0$ is a partition of $K$, $\diam(\V_0) \leq \delta_0$.
  \Loop  
    \State $\E_k \leftarrow$ Transitions($\V_k \cup W, f$)
    \Comment The possible transitions (extra edges may occur).
    \State $\G_k \leftarrow$ GRAPH$(\V_k \cup W, \E_k$)
    \Comment $\G_k \propto (f, |\V_k \cup W|, \V_k \cup W)$
    \Repeat
      \State ready $\leftarrow$ TRUE
      \ForAll {$v \in \V_k$}
	\If {$v \subseteq U \cup |W|$ \textbf{or} $f(v) \subseteq U \cup |W|$}\label{graph-line-basin-check}
	  \State \textbf{move} $v$ from $\V_k$ to $W$
	  \Comment $v$ is contained in the basin of attraction of $\O$
	  \State ready $\leftarrow$ FALSE
	\EndIf
      \EndFor
    \Until {ready}
    \Comment The remaining vertices are not attracted at this resolution
\algstore{boabreak}
\end{algorithmic}
\end{algorithm}
\begin{algorithm}[H]
\ContinuedFloat
\caption{Inner enclosure of the basin of attraction (continued)}
\begin{algorithmic}[1]
\algrestore{boabreak}
    \If {STOP($k, \V_k, W, \delta_k$)}
    \Comment Some stopping condition
      \State \Return $W$
    \EndIf
      \State $\delta_{k+1} \leftarrow$ $\delta_k / 2$
      \State $\V_{k+1} \leftarrow$ Partition($|\V_k|, \delta_{k+1}$)
      \Comment $\V_{k+1}$ is a partition of $|\V_k|$, $\diam(\V_{k+1}) \leq \delta_{k+1}$
    \State $k \leftarrow k + 1$
  \EndLoop
\EndFunction
\end{algorithmic}
\end{algorithm}

In Algorithm \ref{graph-alg-basin}, we move those partition elements into $W$, that lie inside or are mapped into the initial attracting neighbourhood, or the other elements in $W$. Then we refine our remaining partition and continue our procedure with a new one, that has a diameter  half as big as before. Since at the beginning $W$ was empty, it will only contain sets that are contained in the basin of attraction of $\O$. Thus after each cycle, $|W|$ is an inner enclosure of the basin of attraction of $\O$. We stop our iteration when a stopping condition is satisfied, for example $\delta_k < \Delta$, where $\Delta$ is a small positive number given in advance.

\section{The completion of the proof of Theorem \ref{maintheorem}}
\label{application}
Consider now a parameter value $\alpha$ for $F_\alpha$ such that $0.5 \leq \alpha \leq 1$. In order to prove that the fixed point $\underline{\alpha}$ is globally attracting, we need the following observation: given any starting point $(x_0, y_0)$, the accumulation points of the sequence $\left( F^k_\alpha(x_0, y_0) \right)_{k = 0}^\infty$ are non-wandering points of $F_\alpha$. We want to show that the only non-wandering point of $F_\alpha$ in $\R^2_+$ is the fixed point. We know from Lemma \ref{ricker-lemma-trapping} that it is enough to look for points in $S_i^{(\alpha)}$, $i \in \N_0$. Thus our goal is to prove that
\begin{enumerate}
  \item $S_i^{(\alpha)}$ is entirely in the basin of attraction of $\underline{\alpha}$,
  \item or equivalently, $S_i^{(\alpha)}$ contains exactly one non-wandering point, and that is $\underline{\alpha}$.
\end{enumerate}
Our strategy is to divide the parameter range $[0.5, 1] = [0.5, 0.875] \cup [0.875, 0.999] \cup [0.999, 1]$ into small subintervals $[\alpha]$. The diameter of these subintervals will vary between $10^{-3}, 10^{-4}$ and $10^{-5}$ in practice. For one small parameter interval $[\alpha]$ we shall follow these steps:
\begin{enumerate}
  \item Let $i_0 \geq 1$ be the smallest integer such that $|h_{i_0}^{([\alpha])} - h_{{i_0}+1}^{([\alpha])}| + |g_{i_0}^{([\alpha])} - g_{{i_0}+1}^{([\alpha])}| < 10^{-9}$.
  \item The function \texttt{ConstructRegion$([\alpha])$} returns a rigorous enclosure $[S]$ such that
  $$\bigcup_{\alpha \in [\alpha]} S_{i_0}^{(\alpha)} \subseteq [S].$$
  \item Using Propositions \ref{thm_epsalpha}, \ref{thm37} and \ref{thm22}, the function \texttt{FindAttractionDomain$([\alpha])$} returns an $\varepsilon_0 > 0$ such that $\mathrm{K}(\alpha; \varepsilon_0)$ is contained in the basin of attraction of $\underline{\alpha}$ for every $\alpha \in [\alpha]$.
  \item Enclose rigorously $\cup_{\alpha \in [\alpha]} \NonW{F_\alpha}{[S]} \setminus \{ \underline{\alpha} \})$ by removing parts of $[S]$ that do not contain non-wandering points of $F_\alpha$ or are in the basin of attraction of the fixed point $\underline{\alpha}$ for every $\alpha \in [\alpha]$. We do this by simultaneously checking the criteria from line $\ref{graph-line-nonw-check}$ of Algorithm \ref{graph-alg-nonw} and line $\ref{graph-line-basin-check}$ of Algorithm \ref{graph-alg-basin}. If we obtain an empty enclosure at some step, then we have proved that the fixed point in the given parameter region is globally attracting.
\end{enumerate}
We sum it in the following algorithm:
\begin{algorithm}[H]
\caption{Proving the global stability of $\underline{\alpha}$ for the Ricker-map} \label{ricker-alg-main}
\begin{algorithmic}[1]
\Procedure{Ricker}{$[\alpha], \delta$}
  \State $[S] \leftarrow$ \texttt{ConstructRegion($[\alpha]$)}
  \Comment from (\ref{ricker-eq-trapping})
  \State $\varepsilon_0 \leftarrow$ \texttt{FindAttractionDomain($[\alpha]$)}
  \Comment from Propositions \ref{thm_epsalpha}, \ref{thm37} and \ref{thm22}
  \State $[U] \leftarrow \mathrm{K}([\underline{\alpha}]; \varepsilon_0 - (\alpha^+ - \alpha^-))$
  \State $\V \leftarrow$ Partition($[S], \delta$)
  \Comment $\V$ is a partition of $[S]$, $\diam(\V) \leq \delta$
  \Repeat
    \State $\E \leftarrow$ Transitions($\V$, $F_{[\alpha]}$)
    \Comment The possible transitions (extra edges may occur).
    \State $\G \leftarrow$ GRAPH($\V, \E$)
    \Comment $\G \propto (F_{[\alpha]}, |\V|, \V)$
    \State $T \leftarrow \{ v : v$ is in a directed cycle $\}$
    \Comment with the use of Tarjan's algorithm
    \ForAll {$v \in \V$}
      \If{$v \notin T$ \textbf{or} $v \subseteq [U]$ \textbf{or} $F_{[\alpha]}(v) \subseteq [U]$}\label{ricker-line-main-check}
	\State \textbf{remove} $v$ from $\G$
      \EndIf
    \EndFor
    \State $\delta \leftarrow \delta / 2$
    \State $\V \leftarrow$ Partition($|\V|, \delta$)
  \Until{$|\V| = \emptyset$}
\EndProcedure
\end{algorithmic}
\end{algorithm}
We know that \texttt{FindAttractionDomain$[\alpha]$} returns an $\varepsilon_0 > 0$ such that for every $\alpha \in [\alpha] = [\alpha^-, \alpha^+]$, the set $\textrm{K}(\underline{\alpha}; \varepsilon_0)$ is in the basin of attraction of  $\underline{\alpha}$. Assume that $\alpha^+ - \alpha^- < \varepsilon_0$ and let
\begin{equation*}
  \varepsilon = \varepsilon_0 - (\alpha^+ - \alpha^-).
\end{equation*}
Now $\varepsilon > 0$ and the set $\textrm{K}([\underline{\alpha}]; \varepsilon)$ is in the basin of attraction of $\underline{\alpha}$ for every $\alpha \in [\alpha]$. Observe that, using subintervals with $\alpha^+ - \alpha^- \leq 10^{-3}$ and the $\varepsilon_0$ obtained from Propositions \ref{thm_epsalpha}, \ref{thm37} and \ref{thm22}, $\alpha^+ - \alpha^- < \varepsilon_0$ is satisfied.

After each step in the main cycle in Algorithm \ref{ricker-alg-main}, $|\V|$ is a rigorous enclosure of all the non-wandering points of $F_\alpha$ in $[S] \setminus \{\underline{\alpha}\}$ for every $\alpha \in [\alpha]$. This is easy to see since vertices are removed for two possible reasons which are both checked in line $\ref{ricker-line-main-check}$ of Algorithm \ref{ricker-alg-main}. First if $v \notin T$, then $v$ does not contain non-wandering points for any $F_\alpha$, $\alpha \in [\alpha]$ as we have seen in the proof of the correctness of Algorithm \ref{graph-alg-nonw}. Second if $v \subseteq [U]$ or $F_{[\alpha]}(v) \subseteq [U]$ that is $v$ is inside or mapped into the small attracting neighbourhood of every fixed point. Note that if a vertex is inside the basin of attraction of a fixed point $\underline{\alpha}$, then it cannot contain any other non-wandering point of $F_\alpha$, not even on the boundary. This is a similar criterion to  what we have used in line $\ref{graph-line-basin-check}$ 
of Algorithm \ref{graph-alg-basin}. The difference is that now we remove these vertices, consequently we do not have to collect them into a list. If the procedure ends in finite time, then we have established, that there are no other non-wandering points in $[S]$, thus the fixed point is globally attracting for all $\alpha \in [\alpha]$. We state this as
\begin{proposition}
  \label{ricker-thm-main}
  If Algorithm \ref{ricker-alg-main} ends in finite time with input parameters $([\alpha], 10^{-1})$, that is, after finite number of steps, $|\V| = \emptyset$ is satisfied, then $\underline{\alpha}$ is a globally attracting fixed point of the two dimensional Ricker-map $F_\alpha$ for every $\alpha \in [\alpha]$.
\end{proposition}

We implemented our program in C++, using the CAPD Library \cite{CAPD} for rigorous computations, and the Boost Graph Library \cite{Boost_Graph} for handling the directed graphs. The recursion number in Tarjan's algorithm was very high, therefore we converted it into a sequential program, using virtual stack structures from the Standard Library in order to avoid overflows. Instead of simulating the Ricker-map itself, we used its third iterate, the formula is still compact enough not to cause big overestimation in interval arithmetics and it considerably speeds up the calculations.

As an example, we ran our program for the parameter slice $[0.9, 0.90001]$, with $\delta = 10^{-1}$ as the initial diameter for the partition. We show the evolution of the enclosure during the first 8 iterations on Figure \ref{graph-fig-iteration}.
\begin{figure}[H]
    \centering
    \includegraphics[scale=0.3]{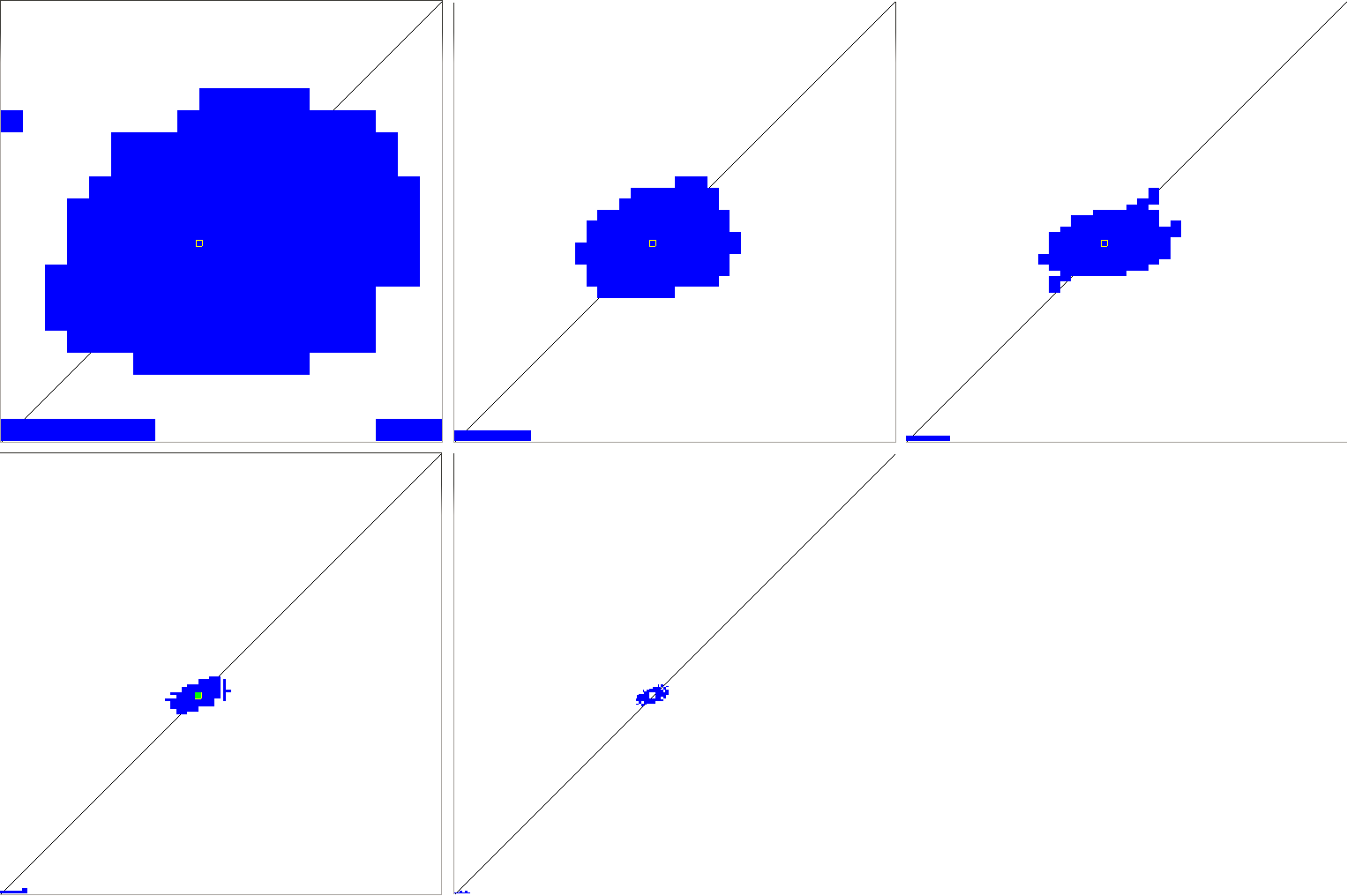}
    \caption{\label{graph-fig-iteration}Enclosure after 1, 3, 4, 6, 8 steps}
\end{figure}
\noindent
The small rectangle is the attracting neighbourhood $[U]$. After 6 iterations the diameter of the partition is sufficiently small in order to have some boxes removed from the inside even though they are in directed cycles. This happens because they are contained in, or get mapped into the basin of attraction of the fixed point.

We used different sizes for the parameter intervals and ran the computations on a cluster of the NIIF HPC centre at the University of Szeged (48 cores, 128 GB memory / cluster) parallelising it with OpenMP. We summarise some technical details in Table \ref{ricker-app-table-resources}.
\begin{table}[H]
\centering
  \begin{tabular}{| l | l | l | l | l | l |}
    \hline
    parameter & size of slices & \# of CPU & max. memory & wall clock time & total time\\ \hline \hline
    $[0.5, 0.875]$ & $10^{-3}$ & $48$ & $2.30$ GB & $1.2$s & $38.5$s \\ \hline
    $[0.875, 0.95]$ & $10^{-3}$ & $48$ & $3.05$ GB & $2.4$s & $52.3$s \\ \hline
    $[0.95, 0.99]$ & $10^{-4}$ & $48$ & $3.07$ GB & $33.7$s & $22$m $5.9$s \\ \hline
    $[0.99, 1]$ & $10^{-5}$ & $20$ & $65.30$ GB & $204$m $42.6$s & $1800$m $37.1$s \\ \hline
  \end{tabular}
\caption{Resources used by the program}
\label{ricker-app-table-resources}
\end{table}
\begin{remark}
  Here \emph{wall clock time (real)} refers to the actual running time of the process, whilst the \emph{total time (user + sys)} is the sum of the time spent on individual CPUs.
\end{remark}
\noindent
The complexity of the computations for some parameter slices is shown in Table \ref{ricker-app-table-complexity}.
\begin{table}[H]
\centering
  \begin{tabular}{| l | l | l | l | l |}
    \hline
    parameter slice & $[S]$ & \# iteration & max \# of vertices & max \# of edges \\ \hline \hline
    $[0.875, 0.876]$ & $[2.072e-04, 5.049031]^2$ & $13$ & $242$ & $1,676$ \\ \hline
    $[0.999, 0.99901]$ & $[2.928e-06, 7.369087]^2$ & $27$ & $729,528$ & $4,193,329$ \\ \hline
    $[0.99999, 1]$ & $[2.822e-06, 7.389015]^2$ & $33$ & $3,105,304$ & $118,751,916$ \\ \hline
  \end{tabular}
\caption{Complexity of the computations}
\label{ricker-app-table-complexity}
\end{table}
\noindent
The program ran successfully, thus we established that the fixed point is globally attracting for $\alpha \in [0.5, 1]$. Combining this with Proposition \ref{ricker-thm-trapping} and Proposition \ref{stable}, the proof of Theorem \ref{maintheorem} is completed.

{\small \section*{Acknowledgements}
\addcontentsline{toc}{section}{Acknowledgements}
The first author was supported by the Bergen Research Foundation. The second and third authors were supported by the Hungarian Scientific Research Fund, Grant No. K 75517. {\'A}bel Garab was also supported by the European Union and co-funded by the European Social Fund. Project title: ``Broadening the knowledge base and supporting the long term professional sustainability of the Research University Centre of Excellence at the University of Szeged by ensuring the rising generation of excellent scientists.'' Project number: T{\'A}MOP-4.2.2/B-10/1-2010-0012.

The computations were performed on the HPC center provided by the Hungarian National Information Infrastructure Development Institute \cite{NIIF} at the University of Szeged.

The authors thank Warwick Tucker from the CAPA group \cite{CAPA}, Daniel Wilczak and Tomasz Kapela, members of the CAPD group \cite{CAPD} for useful suggestions and their help.

\appendix
\section{Appendix}

\begin{dmath}\label{app_g}
g(z,\bar{z},\alpha)=\frac{1}{2 \sqrt{-1+4 \alpha }}\left(-i+\sqrt{-1+4 \alpha }\right) \left(\left(-1+e^{-\frac{z-i z \sqrt{-1+4 \alpha }+\bar{z}+i \sqrt{-1+4 \alpha } \bar{z}}{2 \alpha }}\right) \left(z+\bar{z}\right)+\alpha  \left(-1+e^{-\frac{z-i z \sqrt{-1+4 \alpha }+\bar{z}+i \sqrt{-1+4 \alpha } \bar{z}}{2 \alpha }}+\frac{2 z}{1+i \sqrt{-1+4 \alpha }}+\frac{2 i \bar{z}}{i+\sqrt{-1+4 \alpha }}\right)\right)
\end{dmath}

\begin{equation}\label{app_g20}
g_{20}(\alpha)=-\frac{1}{2}+\frac{3 i}{2 \sqrt{-1+4 \alpha }}
\end{equation}

\begin{equation}
g_{11}(\alpha)=\frac{2 (-1+\alpha )}{-1+4 \alpha +i \sqrt{-1+4 \alpha }}
\end{equation}

\begin{equation}
g_{02}(\alpha)=\frac{i \alpha  \left(3 i+\sqrt{-1+4 \alpha }\right)}{\sqrt{-1+4 \alpha } \left(i-2 i \alpha +\sqrt{-1+4 \alpha }\right)}
\end{equation}

\begin{equation}
g_{30}(\alpha)=-\frac{i \left(1+\alpha -i \sqrt{-1+4 \alpha }\right)}{\alpha  \sqrt{-1+4 \alpha }}
\end{equation}

\begin{equation}
g_{21}(\alpha)=\frac{-3 i+2 i \alpha +\sqrt{-1+4 \alpha }}{2 \alpha  \sqrt{-1+4 \alpha }}
\end{equation}

\begin{equation}
g_{12}(\alpha)=\frac{2 \left(3 i-2 i \alpha +\sqrt{-1+4 \alpha }\right)}{\sqrt{-1+4 \alpha } \left(i+\sqrt{-1+4 \alpha }\right)^2}
\end{equation}

\begin{equation}\label{app_g03}
g_{03}(\alpha)=\frac{4 i \alpha  \left(5 i+\sqrt{-1+4 \alpha }\right)}{\sqrt{-1+4 \alpha } \left(i+\sqrt{-1+4 \alpha }\right)^3}
\end{equation}

\begin{equation}\label{app_g40}
g_{40}(\alpha)=\frac{8 \left(-2+\alpha -2 i \sqrt{-1+4 \alpha }\right)}{\sqrt{-1+4 \alpha } \left(-i+\sqrt{-1+4 \alpha }\right)^3}
\end{equation}

\begin{equation}
g_{31}(\alpha)=-\frac{2 \left(-2 i+i \alpha +\sqrt{-1+4 \alpha }\right)}{\alpha  \left(-i+4 i \alpha +\sqrt{-1+4 \alpha }\right)}
\end{equation}

\begin{equation}
g_{22}(\alpha)=\frac{2 (-2+\alpha )}{\alpha  \left(-1+4 \alpha +i \sqrt{-1+4 \alpha }\right)}
\end{equation}

\begin{equation}
g_{13}(\alpha)=\frac{8 \left(2-\alpha -i \sqrt{-1+4 \alpha }\right)}{\sqrt{-1+4 \alpha } \left(i+\sqrt{-1+4 \alpha }\right)^3}
\end{equation}

\begin{equation}\label{app_g04}
g_{04}(\alpha)=\frac{8 \alpha  \left(7 i+\sqrt{-1+4 \alpha }\right)}{\sqrt{-1+4 \alpha } \left(i+\sqrt{-1+4 \alpha }\right)^4}
\end{equation}

\begin{dgroup*}

\begin{dmath}\label{app_hinv(z)Hij}
h_{inv,\alpha}(z)=
z-\frac{h_{20}(\alpha) z^2}{2}+\frac{1}{6} z^3 \left(3 h_{20}(\alpha)^2-h_{30}(\alpha)+3 h_{11}(\alpha) \overline{h_{02}(\alpha)}\right)+\frac{1}{24} z^4 \left(-15 h_{20}(\alpha)^3+10 h_{20}(\alpha) h_{30}(\alpha)-30 h_{11}(\alpha) h_{20}(\alpha) \overline{h_{02}(\alpha)}-3 h_{02}(\alpha) \overline{h_{02}(\alpha)}^2
+4 h_{11}(\alpha) \overline{h_{03}(\alpha)}-12 h_{11}(\alpha) \overline{h_{02}(\alpha)} \overline{h_{11}(\alpha)}\right)-h_{11}(\alpha) z \bar{z}+\frac{1}{2} z^2 \left(3 h_{11}(\alpha) h_{20}(\alpha)+h_{02}(\alpha) \overline{h_{02}(\alpha)}+2 h_{11}(\alpha) \overline{h_{11}(\alpha)}\right) \bar{z}+\frac{1}{6} z^3 \left(-15 h_{11}(\alpha) h_{20}(\alpha)^2+4 h_{11}(\alpha) h_{30}(\alpha)-12 h_{11}(\alpha)^2 \overline{h_{02}(\alpha)}+3 h_{12}(\alpha) \overline{h_{02}(\alpha)}-6 h_{02}(\alpha) h_{20}(\alpha) \overline{h_{02}(\alpha)}+h_{02}(\alpha) \overline{h_{03}(\alpha)}-12 h_{11}(\alpha) h_{20}(\alpha) \overline{h_{11}(\alpha)}-6 h_{02}(\alpha) \overline{h_{02}(\alpha)} \overline{h_{11}(\alpha)}-6 h_{11}(\alpha) \overline{h_{11}(\alpha)}^2+3 h_{11}(\alpha) \overline{h_{12}(\alpha)}-3 h_{11}(\alpha) \overline{h_{02}(\alpha)} \overline{h_{20}(\alpha)}\right) \bar{z}-\frac{h_{02}(\alpha) \bar{z}^2}{2}+\frac{1}{2} z \left(2 h_{11}(\alpha)^2-h_{12}(\alpha)+h_{02}(\alpha) h_{20}(\alpha)+2 h_{02}(\alpha) \overline{h_{11}(\alpha)}+h_{11}(\alpha) \overline{h_{20}(\alpha)}\right) \bar{z}^2+\frac{1}{4} z^2 \left(-12 h_{11}(\alpha)^2 h_{20}(\alpha)+3 h_{12}(\alpha) h_{20}(\alpha)-3 h_{02}(\alpha) h_{20}(\alpha)^2+h_{02}(\alpha) h_{30}(\alpha)+h_{03}(\alpha) \overline{h_{02}(\alpha)}-9 h_{02}(\alpha) h_{11}(\alpha) \overline{h_{02}(\alpha)}-12 h_{11}(\alpha)^2 \overline{h_{11}(\alpha)}+4 h_{12}(\alpha) \overline{h_{11}(\alpha)}-6 h_{02}(\alpha) h_{20}(\alpha) \overline{h_{11}(\alpha)}-6 h_{02}(\alpha) \overline{h_{11}(\alpha)}^2+2 h_{02}(\alpha) \overline{h_{12}(\alpha)}-3 h_{11}(\alpha) h_{20}(\alpha) \overline{h_{20}(\alpha)}-3 h_{02}(\alpha) \overline{h_{02}(\alpha)} \overline{h_{20}(\alpha)}-6 h_{11}(\alpha) \overline{h_{11}(\alpha)} \overline{h_{20}(\alpha)}\right) \bar{z}^2+\frac{1}{6} \left(-h_{03}(\alpha)+3 h_{02}(\alpha) h_{11}(\alpha)+3 h_{02}(\alpha) \overline{h_{20}(\alpha)}\right) \bar{z}^3+\frac{1}{6} z \left(-6 h_{11}(\alpha)^3+6 h_{11}(\alpha) h_{12}(\alpha)+h_{03}(\alpha) h_{20}(\alpha)-9 h_{02}(\alpha) h_{11}(\alpha) h_{20}(\alpha)-3 h_{02}(\alpha)^2 \overline{h_{02}(\alpha)}+3 h_{03}(\alpha) \overline{h_{11}(\alpha)}-18 h_{02}(\alpha) h_{11}(\alpha) \overline{h_{11}(\alpha)}-6 h_{11}(\alpha)^2 \overline{h_{20}(\alpha)}+3 h_{12}(\alpha) \overline{h_{20}(\alpha)}-3 h_{02}(\alpha) h_{20}(\alpha) \overline{h_{20}(\alpha)}-12 h_{02}(\alpha) \overline{h_{11}(\alpha)} \overline{h_{20}(\alpha)}-3 h_{11}(\alpha) \overline{h_{20}(\alpha)}^2+h_{11}(\alpha) \overline{h_{30}(\alpha)}\right) \bar{z}^3+\frac{1}{24} \left(4 h_{03}(\alpha) h_{11}(\alpha)-12 h_{02}(\alpha) h_{11}(\alpha)^2+6 h_{02}(\alpha) h_{12}(\alpha)-3 h_{02}(\alpha)^2 h_{20}(\alpha)-12 h_{02}(\alpha)^2 \overline{h_{11}(\alpha)}+6 h_{03}(\alpha) \overline{h_{20}(\alpha)}-18 h_{02}(\alpha) h_{11}(\alpha) \overline{h_{20}(\alpha)}-15 h_{02}(\alpha) \overline{h_{20}(\alpha)}^2+4 h_{02}(\alpha) \overline{h_{30}(\alpha)}\right) \bar{z}^4
\end{dmath}
\end{dgroup*}

\begin{equation}\label{app_h20}
h_{20}(\alpha)=\frac{2 \left(1+\sqrt{1-4 \alpha }-\alpha \right)}{\alpha  \left(1-4 \alpha +i \sqrt{-1+4 \alpha }\right)}
\end{equation}

\begin{equation}\label{app_h11}
h_{11}(\alpha)=\frac{2 (-1+\alpha )}{\alpha  \left(-1+4 \alpha -i \sqrt{-1+4 \alpha }\right)}
\end{equation}

\begin{equation}\label{app_h02}
h_{02}(\alpha)=\frac{2 \alpha  \left(3 i+\sqrt{-1+4 \alpha }\right)}{\left(i+\sqrt{-1+4 \alpha }\right)^2 \left(-i+4 i \alpha +\alpha  \sqrt{-1+4 \alpha }\right)}
\end{equation}

\begin{dgroup*}
\begin{dmath*}
h_{30}(\alpha)=
\end{dmath*}
\begin{dmath}[breakdepth={2}]\label{app_h30}
\left(16 \left(-12-12 i \sqrt{-1+4 \alpha }+\alpha  \left(45+21 i \sqrt{-1+4 \alpha }+\alpha  \left(-27+2 \alpha -5 i \sqrt{-1+4 \alpha }\right)\right)\right)\right)\\
\left(\left(1+\sqrt{1-4 \alpha }\right)^4 \sqrt{-1+4 \alpha } \left(i+\sqrt{-1+4 \alpha}+\alpha  \left(-2 i+\alpha  \left(-5 i+\sqrt{-1+4 \alpha }\right)\right)\right)\right)^{-1}
\end{dmath}
\end{dgroup*}

\begin{dgroup*}
\begin{dmath*}
h_{12}(\alpha)=
\end{dmath*}
\begin{dmath}[breakdepth={2}]\label{app_h12}
\left(\left(4 i+\alpha  \left(-4 \left(5 i+3 \sqrt{-1+4 \alpha }\right)+\alpha  \left(28 i+18 \sqrt{-1+4 \alpha }+\alpha  \left(-31 i-9 \sqrt{-1+4 \alpha }+\\
\alpha  \left(9 i+\sqrt{-1+4 \alpha }\right)\right)\right)\right)\right)\right)
\left(\left(\alpha ^3 \left(-i+\sqrt{-1+4 \alpha }+\alpha  \left(6 i+\alpha  \left(-13 i-15 \sqrt{-1+4 \alpha }+\\
2 \alpha  \left(10 i+\sqrt{-1+4 \alpha }\right)\right)\right)\right)\right)\right)^{-1}
\end{dmath}
\end{dgroup*}

\begin{dgroup*}
\begin{dmath*}
h_{03}(\alpha)=
\end{dmath*}
\begin{dmath}[breakdepth={2}]\label{app_h03}
\left(\left(2 \left(5 i+7 \sqrt{-1+4 \alpha }+\alpha  \left(-29 i-14 \sqrt{-1+4 \alpha }+\alpha  \left(17 i-i \alpha +3 \sqrt{-1+4 \alpha }\right)\right)\right)\right)\right)\\
\left(\left(i+\sqrt{-1+4 \alpha }+\alpha  \left(-2 \left(7 i+6 \sqrt{-1+4 \alpha }\right)+\alpha  \left(70 i+46 \sqrt{-1+4 \alpha }+\\
\alpha  \left(-6 \left(23 i+10 \sqrt{-1+4 \alpha }\right)+\alpha  \left(73 i-4 i \alpha +13 \sqrt{-1+4 \alpha }\right)\right)\right)\right)\right)\right)^{-1}
\end{dmath}
\end{dgroup*}

\begin{dgroup*}
\begin{dmath*}
\left|h_{inv}^{20}(\alpha)\right|+\left|h_{inv}^{11}(\alpha)\right|+\left|h_{inv}^{02}(\alpha)\right|=
\end{dmath*}
\begin{dmath}\label{app_hinv2}
\left(-2 (-1+\alpha ) \alpha +\sqrt{\alpha ^3 (2+\alpha )}+\sqrt{\frac{\alpha ^5 (2+\alpha )}{-1+\alpha  (4+\alpha )}}\right)\left(2 \sqrt{\alpha ^5 (-1+4 \alpha )}\right)^{-1}
\end{dmath}
\end{dgroup*}

\begin{dgroup*}
\begin{dmath}\label{app_hinv3}
\left|h_{inv}^{30}(\alpha)\right|+\left|h_{inv}^{21}(\alpha)\right|+\left|h_{inv}^{12}(\alpha)\right|+
\left|h_{inv}^{03}(\alpha)\right|=\\
\end{dmath}
\begin{dmath*}
\frac{1}{6} \sqrt{\frac{36+15 \alpha -12 \alpha ^2+4 \alpha ^3}{\alpha ^5 \left(-2+7 \alpha +4 \alpha ^2\right)}}
+
\sqrt{\frac{6-40 \alpha +75 \alpha ^2-21 \alpha ^4+4 \alpha ^5+4 \alpha ^6}{(2-15 \alpha +22 \alpha ^2+23 \alpha ^3+4 \alpha ^4)4\alpha^6}}
+
\frac{1}{6} \sqrt{\frac{12-54 \alpha +66 \alpha ^2-13 \alpha ^3+4 \alpha ^4+4 \alpha ^5}{\alpha ^3 \left(-1+13 \alpha -57 \alpha ^2+88 \alpha ^3-18 \alpha ^4+7 \alpha ^5+4 \alpha ^6\right)}}
+
\frac{1}{2} \sqrt{\frac{2-16 \alpha +13 \alpha ^2+124 \alpha ^3-198 \alpha ^4+15 \alpha ^5+54 \alpha ^6+9 \alpha ^7}{\alpha ^6 (-1+4 \alpha ) \left(-1+4 \alpha +\alpha ^2\right)^2}}
\end{dmath*}
\end{dgroup*}

\begin{dgroup*}
\begin{dmath*}
|h_{inv}^{40}(\alpha)|=
\end{dmath*}
\begin{dmath}
\frac{1}{24} \left(-1140+10874 \alpha -12660 \alpha ^2-143073 \alpha ^3+423177 \alpha ^4-211261 \alpha ^5+1356 \alpha ^6+30869 \alpha ^7-13766 \alpha ^8-2238 \alpha ^9+1489 \alpha ^{10}+256 \alpha ^{11}\right)^{\frac{1}{2}}
\left((1-4 \alpha )^2 \alpha ^6 \left(-1+4 \alpha +\alpha ^2\right)^3 \left(-1+5 \alpha -2 \alpha ^2+\alpha ^3\right)\right)^{-\frac{1}{2}}
\end{dmath}
\end{dgroup*}

\begin{dgroup*}
\begin{dmath*}
|h_{inv}^{31}(\alpha)|=
\end{dmath*}
\begin{dmath}
\left(-36+912 \alpha -9354 \alpha ^2+49026 \alpha ^3-133548 \alpha ^4+155248 \alpha ^5+24851 \alpha ^6-182342 \alpha ^7+127014 \alpha ^8-47122 \alpha ^9-12543 \alpha ^{10}+34528 \alpha ^{11}-329 \alpha ^{12}+1925 \alpha ^{13}+2452 \alpha ^{14}+361 \alpha ^{15}\right)^{\frac{1}{2}}
\left(-1+5 \alpha -2 \alpha ^2+\alpha ^3\right)^{-\frac{1}{2}}
\left(6 \alpha ^5 \left(2-15 \alpha +22 \alpha ^2+23 \alpha ^3+4 \alpha ^4\right)\right)^{-1}
\end{dmath}
\end{dgroup*}

\begin{dgroup*}
\begin{dmath*}
|h_{inv}^{22}(\alpha)|=
\end{dmath*}
\begin{dmath}
\left(-16+4 \alpha +2240 \alpha ^2-14868 \alpha ^3+19782 \alpha ^4+91297 \alpha ^5-309731 \alpha ^6+259610 \alpha ^7+81080 \alpha ^8-147591 \alpha ^9+12815 \alpha ^{10}+19871 \alpha ^{11}-8236 \alpha ^{12}+241 \alpha ^{13}+2132 \alpha ^{14}+361 \alpha ^{15}\right)
\left(-1+5 \alpha -2 \alpha ^2+\alpha ^3\right)^{-\frac{1}{2}}
\left(4 \alpha ^5 \left(2-15 \alpha +22 \alpha ^2+23 \alpha ^3+4 \alpha ^4\right)\right)^{-1}
\end{dmath}
\end{dgroup*}

\begin{dgroup*}
\begin{dmath*}
|h_{inv}^{13}(\alpha)|=
\end{dmath*}
\begin{dmath}
\frac{1}{6} \left(\left(36-840 \alpha +7728 \alpha ^2-35454 \alpha ^3+84157 \alpha ^4-96139 \alpha ^5+39017 \alpha ^6+15361 \alpha ^7-22836 \alpha ^8+10489 \alpha ^9+5142 \alpha ^{10}-397 \alpha ^{11}+922 \alpha ^{12}+529 \alpha ^{13}+64 \alpha ^{14}\right)\right)^{\frac{1}{2}}
\left(\alpha ^9 \left(-1+4 \alpha +\alpha ^2\right)^3 \left(2-17 \alpha +35 \alpha ^2+4 \alpha ^3-\alpha ^4+4 \alpha ^5\right)\right)^{-\frac{1}{2}}
\end{dmath}
\end{dgroup*}

\begin{dgroup*}
\begin{dmath*}
|h_{inv}^{04}(\alpha)|=
\end{dmath*}
\begin{dmath}\label{app_hinv_04}
\frac{1}{24} \left(36-396 \alpha +1350 \alpha ^2-1422 \alpha ^3+318 \alpha ^4+441 \alpha ^5-145 \alpha ^6+100 \alpha ^7+25 \alpha ^8\right)^{\frac{1}{2}}
\left(\alpha ^6 \left(-1+4 \alpha +\alpha ^2\right)^2 \left(1-9 \alpha +22 \alpha ^2-9 \alpha ^3+4 \alpha ^4\right)\right)^{-\frac{1}{2}}
\end{dmath}
\end{dgroup*}

\begin{dgroup*}
\begin{dmath*}
|r_2^{40}(\alpha)|=
\end{dmath*}
\begin{dmath}\label{app_r2_40}
\frac{1}{24}\left(-432+2808 \alpha -5016 \alpha ^2+4692 \alpha ^3-2584 \alpha ^4+930 \alpha ^5-210 \alpha ^6+7 \alpha ^7+27 \alpha ^8-6 \alpha ^9+\alpha ^{10}\right)^{\frac{1}{2}}\\
\left(\alpha ^4 \left(-1+4 \alpha +\alpha ^2\right)^2 \left(1-9 \alpha +22 \alpha ^2-9 \alpha ^3+4 \alpha ^4\right)\right)^{-\frac{1}{2}}
\end{dmath}
\end{dgroup*}

\begin{dgroup*}
\begin{dmath*}
|r_2^{31}(\alpha)|=
\end{dmath*}
\begin{dmath}
\frac{1}{6}\left(90-810 \alpha +1635 \alpha ^2+1485 \alpha ^3-3462 \alpha ^4+4056 \alpha ^5-659 \alpha ^6-108 \alpha ^7-320 \alpha ^8+24 \alpha ^9+87 \alpha ^{10}-102 \alpha ^{11}+36 \alpha ^{12}+\alpha ^{13}\right)^{\frac{1}{2}}\\
\left(\alpha ^7 \left(2-17 \alpha +33 \alpha ^2+13 \alpha ^3+4 \alpha ^5+\alpha ^6\right)(-1+4\alpha)^2\right)^{-\frac{1}{2}}
\end{dmath}
\end{dgroup*}

\begin{dgroup*}
\begin{dmath*}
|r_2^{22}(\alpha)|=
\end{dmath*}
\begin{dmath}
\frac{1}{4}\left(64-608 \alpha +1348 \alpha ^2+728 \alpha ^3-1692 \alpha ^4+1948 \alpha ^5-688 \alpha ^6-2008 \alpha ^7+1184 \alpha ^8+742 \alpha ^9-355 \alpha ^{10}-136 \alpha ^{11}+34 \alpha ^{12}+14 \alpha ^{13}+\alpha ^{14}\right)^{\frac{1}{2}}\\
\left(\alpha ^9 (-1+4 \alpha ) \left(-2+7 \alpha +6 \alpha ^2+\alpha ^3\right)^2\right)^{-\frac{1}{2}}
\end{dmath}
\end{dgroup*}

\begin{dgroup*}
\begin{dmath*}
|r_2^{13}(\alpha)|=
\end{dmath*}
\begin{dmath}
\frac{1}{6}\left(-54+918 \alpha -5973 \alpha ^2+18921 \alpha ^3-32250 \alpha ^4+32742 \alpha ^5-15643 \alpha ^6+2506 \alpha ^7+3246 \alpha ^8-3551 \alpha ^9+1327 \alpha ^{10}-156 \alpha ^{11}-48 \alpha ^{12}+59 \alpha ^{13}+16 \alpha ^{14}+\alpha ^{15}\right)^{\frac{1}{2}}\\
\left(\alpha ^7 \left(-2+9 \alpha +\alpha ^2+\alpha ^4\right)(6-48 \alpha +90 \alpha ^2+4 \alpha ^3)^2\right)^{-\frac{1}{2}}
\end{dmath}
\end{dgroup*}

\begin{dgroup*}
\begin{dmath*}
|r_2^{40}(\alpha)|=
\end{dmath*}
\begin{dmath}\label{app_r2_04}
\frac{1}{24}\left(-432+2808 \alpha -5016 \alpha ^2+4692 \alpha ^3-2584 \alpha ^4+930 \alpha ^5-210 \alpha ^6+7 \alpha ^7+27 \alpha ^8-6 \alpha ^9+\alpha ^{10}\right)^{\frac{1}{2}}\\
\left(\alpha ^4 \left(-1+4 \alpha +\alpha ^2\right)^2 \left(1-9 \alpha +22 \alpha ^2-9 \alpha ^3+4 \alpha ^4\right)\right)^{-\frac{1}{2}}
\end{dmath}
\end{dgroup*}

\begin{dgroup*}
\begin{dmath}\label{app_c1}
c_1(\alpha)=
\end{dmath}
\begin{dmath*}
\left(2 i-2 \sqrt{-1+4 \alpha }+\alpha  \left(2 \left(-7 i+5 \sqrt{-1+4 \alpha }\right)+\alpha  \left(25 i-13 \sqrt{-1+4 \alpha }+\alpha  \left(-25 i+7 \sqrt{-1+4 \alpha }-\alpha  \left(-7 i+\sqrt{-1+4 \alpha }\right)\right)\right)\right)\right)\\
\left(2 \alpha ^3 \sqrt{-1+4 \alpha } \left(-i+\sqrt{-1+4 \alpha }\right) \left(i \alpha +\sqrt{-1+4 \alpha }\right)\right)^{-1}
\end{dmath*}
\end{dgroup*}

}

\bibliographystyle{plain}
\bibliography{ricker2012}
\addcontentsline{toc}{section}{References}

\end{document}